\documentclass[12pt,reqno]{amsart}
\usepackage{amssymb}
\usepackage[all,cmtip]{xy}
\usepackage[dvips]{graphics}
\usepackage{epsfig}
\usepackage{subfigure}
\usepackage{hyperref}

\newcommand{\Q}{\mathbb Q}
\newcommand{\R}{\mathbb R}
\newcommand{\Z}{\mathbb Z}

\renewcommand{\P}{\mathbb P}

\DeclareMathOperator{\Aut}{Aut}
\DeclareMathOperator{\Vor}{Vor}
\DeclareMathOperator{\rank}{rank}

\newtheorem{theorem}{Theorem}[section]
\newtheorem{proposition}[theorem]{Proposition}
\newtheorem{construction}[theorem]{Construction}

\newtheorem{lemma}[theorem]{Lemma}

\theoremstyle{definition}
\newtheorem{definition}[theorem]{Definition}
\newtheorem{example}[theorem]{Example}
\newtheorem{algorithm}[theorem]{Algorithm}
\theoremstyle{remark}
\newtheorem{remark}[theorem]{Remark}

\newcommand{\Atrg}{\ensuremath{A^{\text{tr}}_g}}
\newcommand{\Atr}[1]{\ensuremath{A^{\text{tr}}_#1}}
\newcommand{\Mtrg}{\ensuremath{M^{\text{tr}}_g}}
\newcommand{\ttrg}{\ensuremath{t^{\text{tr}}_g}}
\newcommand{\Azong}{\ensuremath{A^{\text{zon}}_g}}
\newcommand{\Acogrg}{\ensuremath{A^{\text{cogr}}_g}}
\newcommand{\Acogr}[1]{\ensuremath{A^{\text{cogr}}_#1}}
\newcommand{\hati}{{\ensuremath{{\hat{\imath}}}}}
\newcommand{\hatj}{{\ensuremath{{\hat{\jmath}}}}}
\newcommand{\mm}[4]{\ensuremath{\left(\begin{matrix}#1 & #2\\ #3 & #4\end{matrix}\right)}}
\DeclareMathOperator{\Del}{Del}
\DeclareMathOperator{\Stab}{Stab}
\DeclareMathOperator{\sspan}{span}
\DeclareMathOperator{\cone}{cone}
\DeclareMathOperator{\Out}{Out}

\newcommand{\hooklongrightarrow}{\lhook\joinrel\longrightarrow}
\newcommand{\twoheadlongrightarrow}{\relbar\joinrel\twoheadrightarrow}
\title{Combinatorics of the tropical Torelli map}
\author{Melody Chan}
\email{mtchan@math.berkeley.edu}
\address{Department of Mathematics, University of California, Berkeley}

\date{\today}

\begin{document}

\begin{abstract}This paper is a combinatorial and computational study of the moduli space $\Mtrg$ of tropical curves of genus $g$, the moduli space $\Atrg$ of principally polarized tropical abelian varieties, and the tropical Torelli map.  These objects were studied recently by Brannetti, Melo, and Viviani.  Here, we give a new definition of the category of stacky fans, of which $\Mtrg$ and $\Atrg$ are objects and the Torelli map is a morphism.  We compute the poset of cells of $\Mtrg$ and of the tropical Schottky locus for genus at most 5.  
 We show that $\Atrg$ is Hausdorff, and we also construct a finite-index cover for the space $\Atr{3}$ which satisfies a tropical-type balancing condition.  Many different combinatorial objects, 
 including regular matroids, positive semidefinite forms, and metric~graphs, play~a~role.
\end{abstract}

\maketitle

\tableofcontents 

\section{Introduction}

This paper is a combinatorial and computational study of the tropical moduli spaces $\Mtrg$ and $\Atrg$ and the tropical Torelli map.

There is, of course, a vast (to say the least) literature on the subjects of algebraic curves and moduli spaces in algebraic geometry. For example, two well-studied objects are the moduli space $\mathcal{M}_g$ of smooth projective complex curves of genus $g$ and the moduli space $\mathcal{A}_g$ of $g$-dimensional principally polarized abelian varieties. The Torelli map
\[ t_g: \mathcal{M}_g \rightarrow \mathcal{A}_g \]
then sends a genus $g$ algebraic curve to its Jacobian, which is a certain $g$-dimensional complex torus. The image of $t_g$ is called the Torelli locus or the Schottky locus. The problem of how to characterize the Schottky locus inside $\mathcal{A}_g$ is already very deep.  See, for example, the survey of Grushevsky \cite{gr}.

The perspective we take in this paper is the perspective of tropical geometry \cite{ms}. From this viewpoint, one replaces algebraic varieties with piecewise-linear or polyhedral objects. These latter objects are amenable to combinatorial techniques, but they still carry information about the former ones. Roughly speaking, the information they carry has to do with what is happening ``at the boundary'' or ``at the missing points'' of the algebraic object. 

For example, the tropical analogue of $\mathcal{M}_g$, denoted $\Mtrg$, parametrizes certain weighted metric graphs, and it has a poset of cells corresponding to the boundary strata of the Deligne-Mumford compactification $\overline{\mathcal{M}_g}$ of $\mathcal{M}_g$.  Under this correspondence, a stable curve $C$ in $\overline{\mathcal{M}_g}$ is sent to its so-called dual graph.  The irreducible components of $C$, weighted by their geometric genus, are the vertices of this graph, and each node in the intersection of two components is recorded with an edge.
The correspondence in genus $2$ is shown in Figure \ref{f:m_2}.
A rigorous proof of this correspondence was given by Caporaso in \cite[Section 5.3]{c}.

We remark that the correspondence above yields dual graphs that are just graphs, not metric graphs.  There is not yet in the literature a sensible way to equip these graphs with edge lengths and thus produce an actual tropicalization map $\overline{\mathcal{M}_g}\rightarrow\Mtrg$.  So for now, the correspondence between $\overline{\mathcal{M}_g}$ and $\Mtrg$ is not as tight as it ultimately should be.  However, we anticipate that forthcoming work on Berkovich spaces by Baker, Payne, and Rabinoff will address this interesting point.

The starting point of this paper is the recent paper by Brannetti, Melo, and Viviani \cite{bmv}. In that paper, the authors rigorously define a plausible category for tropical moduli spaces called stacky fans. (The term ``stacky fan'' is due to the authors of \cite{bmv}, and is unrelated, as far as we know, to the construction of Borisov, Chen, and Smith in \cite{bcs}).  They further define the tropical versions $\Mtrg$ and $\Atrg$ of $\mathcal{M}_g$ and $\mathcal{A}_g$ and a tropical Torelli map between them, and prove many results about these objects, some of which we will review here.

Preceding that paper is the foundational work of Mikhalkin in \cite{mapplications} and of Mikhalkin and Zharkov \cite{mz}, in which tropical curves and Jacobians were first introduced and studied in detail.  The notion of tropical curves in \cite{bmv} is slightly different from the original definition, in that curves now come equipped with vertex weights.
We should also mention the work of Caporaso \cite{c}, who proves geometric results on $\Mtrg$ considered just as a topological space, and Caporaso and Viviani \cite{cv}, who prove a tropical Torelli theorem stating that the tropical Torelli map is ``mostly'' injective, as originally conjectured in \cite{mz}.

In laying the groundwork for the results we will present here, we ran into some inconsistencies in \cite{bmv}. It seems that the definition of a stacky fan there is inadvertently restrictive. In fact, it excludes $\Mtrg$ and $\Atrg$ themselves from being stacky fans. Also, there is a topological subtlety in defining $\Atrg$, which we will address in \S 4.4. Thus, we find ourself doing some foundational work here too. 

We begin in Section 2 by recalling the definition in \cite{bmv} of the tropical moduli space $\Mtrg$ and presenting computations, summarized in Theorem \ref{t:comp}, for $g \leq 5$.
With $\Mtrg$ as a motivating example, we attempt a better definition of stacky fans in Section 3. In Section 4, we define the space $\Atrg$, recalling the beautiful combinatorics of Voronoi decompositions along the way, and prove that it is Hausdorff.  Note that our definition of this space, Definition \ref{d:ag}, is different from the one in \cite[Section 4.2]{bmv}, and it corrects a minor error there.  In Section 5, we study the combinatorics of the zonotopal subfan.  We review the tropical Torelli map in Section 6; Theorem \ref{t:sch} presents computations on the tropical Schottky locus for $g \leq 5$. Tables 1 and 2 compare the number of cells in the stacky fans $\Mtrg$, the Schottky locus, and $\Atrg$ for $g \leq 5$. In Section 7, we partially answer a question suggested by Diane Maclagan: we give finite-index covers of $\Atr{2}$ and $\Atr{3}$ that satisfy a tropical-type balancing condition.

\medskip

{\bf Acknowledgments.}  The author thanks B.~Sturmfels, D.~Maclagan, and F.~Vallentin for helpful discussions, M.~Melo and F.~Viviani for comments on an earlier draft, F.~Vallentin for many useful references, K.~Vogtmann for the reference to \cite{br}, F.~Shokrieh for insight on Delone subdivisions, and R.~Masuda for much help with typesetting.  The author is supported by a Graduate Research Fellowship from the National Science Foundation.

\begin{figure}[hbtp]%
\includegraphics[width=\columnwidth,angle=0,scale=1.1]{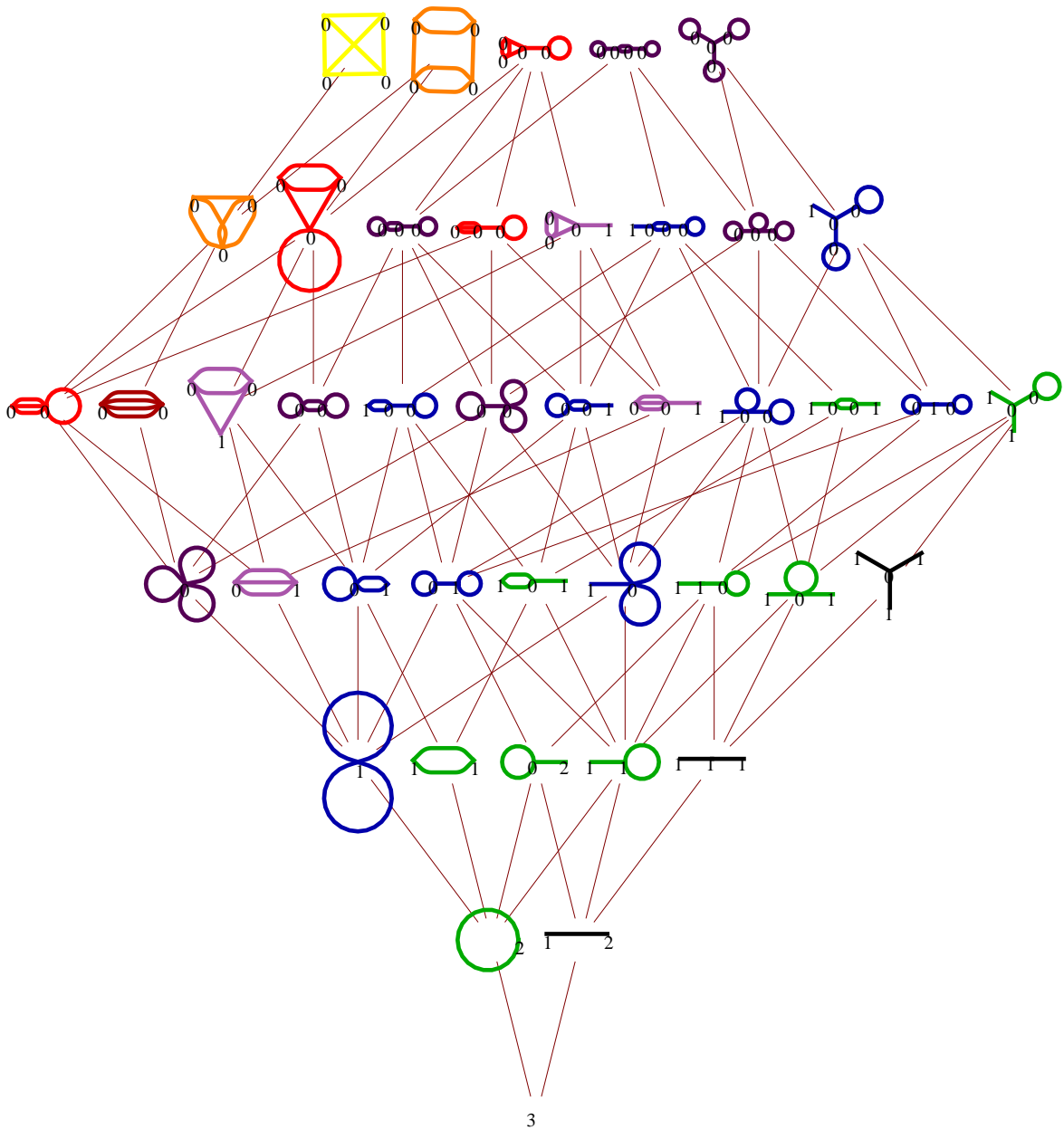}%
\vskip -.7in
\caption{Poset of cells of $M^{\text{tr}}_3$, color-coded according to their images in $A^{\text{tr}}_3$ via the tropical Torelli map.}%
\label{f:m_3}%
\end{figure}

\section{The moduli space of tropical curves}

In this section, we review the construction in \cite{bmv} of the moduli space of tropical curves of a fixed genus $g$ (see also \cite{mapplications}). This space is denoted $\Mtrg$. Then, we present explicit computations of these spaces in genus up to~$5$.

We will see that the moduli space $\Mtrg$ is not itself a tropical variety, in that it does not have the structure of a balanced polyhedral fan (\cite[Definition 3.3.1]{ms}).  That would be too much to expect, as it has automorphisms built in to its structure that precisely give rise to ``stackiness.''  Contrast this with the situation of moduli space $M_{0,n}$ of tropical rational curves with $n$ marked points, constructed and studied in \cite{gkm}, \cite{mrational}, and \cite{ss}.  As expected by analogy with the classical situation, this latter space is well known to have the structure of a tropical variety that comes from the tropical Grassmannian $Gr(2,n)$.

\subsection{Definition of tropical curves.}

Before constructing the moduli space of tropical curves, let us review the definition of a tropical curve. 

First, recall that a {\bf metric graph} is a pair $(G,l)$, where $G$ is a finite connected graph, loops and parallel edges allowed, and $l$ is a function
\[ l: E(G) \rightarrow \R_{>0} \]
on the edges of $G$. We view $l$ as recording lengths of the edges of $G$. The {\bf genus} of a graph $G$ is the rank of its first homology group:
\[ g(G) = |E|-|V|+1. \]

\begin{definition}
  A {\bf tropical curve} $C$ is a triple $(G,l,w)$, where $(G,l)$ is a metric graph (so $G$ is connected), and $w$ is a weight function
  \[ w:V(G) \rightarrow \Z_{\geq 0} \]
  on the vertices of $G$, with the property that every weight zero vertex has degree at least 3.  
\end{definition}

\begin{definition}
Two tropical curves $(G,l,w)$ and $(G',l',w')$ are isomorphic if there is an isomorphism of graphs $G \xrightarrow{\cong} G'$ that preserves edge lengths and preserves vertex weights. 
\end{definition}

We are interested in tropical curves only up to isomorphism.
When we speak of a tropical curve, we will really mean its isomorphism class.

\begin{definition}
\noindent Given a tropical curve $C = (G,l,w)$, write
\[ |w| := \sum_{v \in V(G)} w(v). \]
Then the {\bf genus} of $C$ is defined to be
\[ g(C) = g(G) + |w|. \]
The {\bf combinatorial type} of $C$ is the pair $(G,w)$, in other words, all of the data of $C$ except for the edge lengths.
\end{definition}

\begin{remark}
Informally, we view a weight of $k$ at a vertex $v$ as $k$ loops, based at $v$, of infinitesimally small length. Each infinitesimal loop contributes once to the genus of $C$. Furthermore, the property that only vertices with positive weight may have degree 1 or 2 amounts to requiring that, were the infinitesimal loops really to exist, every vertex would have degree at least 3.

Permitting vertex weights will ensure that the moduli space $\Mtrg$, once it is constructed, is complete. That is, a sequence of genus $g$ tropical curves obtained by sending the length of a loop to zero will still converge to a genus $g$ curve.
Of course, the real reason to permit vertex weights is so that the combinatorial types of genus $g$ tropical curves correspond precisely to dual graphs of stable curves in $\overline{\mathcal{M}_g}$, as discussed in the introduction and in \cite[Section 5.3]{c}.  See Figure \ref{f:m_2}.
\end{remark}

\begin{figure}%
\includegraphics[width=3.5in]{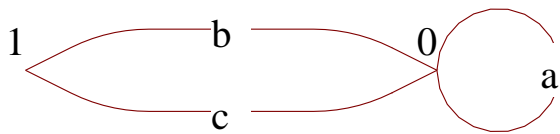}%
\caption{A tropical curve of genus 3.  Here, $a,b,c$ are fixed positive real numbers.}%
\label{f:curve}%
\end{figure}

Figure \ref{f:curve} shows an example of a tropical curve $C$ of genus 3. Note that if we allow the edge lengths $l$ to vary over all positive real numbers, we obtain all tropical curves of the same combinatorial type as $C$. This motivates our construction of the moduli space of tropical curves below. We will first group together curves of the same combinatorial type, obtaining one cell for each combinatorial type. Then, we will glue our cells together to obtain the moduli space.

\subsection{Definition of the moduli space of tropical curves}
Fix $g \geq 2$. Our goal now is to construct a moduli space for genus $g$ tropical curves, that is, a space whose points correspond to tropical curves of genus $g$ and whose geometry reflects the geometry of the tropical curves in a sensible way. The following construction is due to the authors of~\cite{bmv}.

First, fix a combinatorial type $(G,w)$ of genus $g$. What is a parameter space for all tropical curves of this type? 
Our first guess might be a positive orthant $\R_{>0}^{|E(G)|}$, that is, a choice of positive length for each edge of $G$.
But we have overcounted by symmetries of the combinatorial type $(G,w)$. For example, in Figure \ref{f:curve}, $(a,b,c) = (1,2,3)$ and $(a,b,c)=(1,3,2)$ give the same tropical curve. 

Furthermore, with foresight, we will allow lengths of zero on our edges as well, with the understanding that a curve with some zero-length edges will soon be identified with the curve obtained by contracting those edges. This suggests the following definition.

\begin{definition}
  Given a combinatorial type  $(G,w)$, let
  the {\bf automorphism group} $\Aut(G,w)$ be the set of all permutations $\varphi: E(G) \to E(G)$ that arise from weight-preserving automorphisms of $G$. That is, $\Aut(G,w)$ is the set of permutations $\varphi:E(G) \to E(G)$ that admit a permutation $\pi: V(G) \to V(G)$ which preserves the weight function $w$, and such that if an edge $e \in E(G)$ has endpoints $v$ and $w$, then $\varphi(e)$ has endpoints $\pi(v)$ and $\pi(w)$.

  Now, the group $\Aut(G,w)$ acts naturally on the set $E(G)$, and hence on the 
  orthant $\R_{\geq 0}^{E(G)}$, with the latter action given by permuting coordinates. We define $\overline{C(G,w)}$ to be the topological quotient space
  \[ \overline{C(G,w)} = \frac{\R_{\geq 0}^{E(G)}}{\Aut(G,w)}. \]
\end{definition}

Next, we define an equivalence relation on the points in the union $$\coprod \overline{C(G,w)},$$ as $(G,w)$ ranges over all combinatorial types of genus $g$. Regard a point $x \in \overline{C(G,w)}$ as an assignment of lengths to the edges of $G$. Now, given two points $x \in \overline{C(G,w)}$ and $x' \in \overline{C(G',w')}$, identify $x$ and $x'$ if one of them is obtained from the other by contracting all edges of length zero. Note that contracting a loop, say at vertex $v$, means deleting that loop and adding 1 to the weight of $v$. Contracting a nonloop edge, say with endpoints $v_1$ and $v_2$, means deleting that edge and identifying $v_1$ and $v_2$ to obtain a new vertex whose weight is $w(v_1) + w(v_2)$. 

Finally, let $\sim$ be the smallest equivalence relation containing the identification we have just defined. Now we glue the cells $\overline{C(G,w)}$ along $\sim$ to obtain our moduli space:

\begin{definition}
  The {\bf moduli space} $\Mtrg$ is the topological space
  \[ \Mtrg \,\,:= \,\,\coprod \overline{C(G,w)} /\!\sim, \]
  where the disjoint union ranges over all combinatorial types of genus $g$, and $\sim$ is the equivalence relation defined above.
\end{definition}

In fact, the space $\Mtrg$ carries additional structure: it is an example of a stacky fan. We will define the category of stacky fans in Section~3. 
\begin{figure}%
\includegraphics[width=4in]{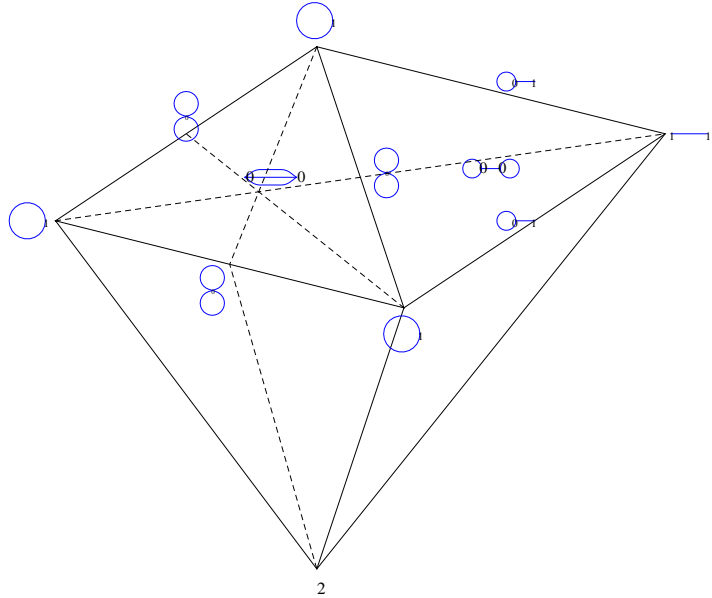}%
\caption{The stacky fan $M^{\text{tr}}_2$.}%
\label{f:m2_fan}%
\end{figure}

\begin{example}
 Figure \ref{f:m2_fan} is a picture of $M^{\text{tr}}_2$.  Its cells are quotients of polyhedral cones; the dotted lines represent symmetries, and faces labeled by the same combinatorial type are in fact identified.  The poset of cells, which we will investigate next for higher $g$, is shown in Figure~\ref{f:m_2}.
 It has two vertices, two edges and two $2$-cells.
\end{example}



\subsection{Explicit computations of $\Mtrg$}

Our next goal will be to compute the space $\Mtrg$ for $g$ at most 5. The computations were done in \textsc{Mathematica}, and the code is available at
$$\texttt{http://math.berkeley.edu/\~{}mtchan/torelli/}$$

What we compute, to be precise, is the partially ordered set $P_g$ on the cells of $\Mtrg$.
This poset is defined in Lemma~\ref{l:pg} below.
Our results, summarized in Theorem \ref{t:comp} below, provide independent verification of the first six terms of the sequence A174224 in \cite{sloane}, which counts the number of tropical curves of genus $g$:
$$0, 0, 7, 42, 379, 4555, 69808, 1281678,...$$  
This sequence, along with much more data along these lines, was first obtained by Maggiolo and Pagani by an algorithm described in \cite{mp}.

\begin{definition}
Given two combinatorial types $(G,w)$ and $(G',w')$ of genus $g$, we say that $(G',w')$ is a {\bf specialization}, or {\bf contraction}, of $(G,w)$, if it can be obtained from $(G,w)$ by a sequence of edge contractions. Here, contracting a loop means deleting it and adding 1 to the weight of its base vertex; contracting a nonloop edge, say with endpoints $v_1$ and $v_2$, means deleting the edge and identifying $v_1$ and $v_2$ to obtain a new vertex whose weight we set to $w(v_1)+w(v_2)$.
\end{definition}

\begin{lemma}\label{l:pg}
  The relation of specialization on genus $g$ combinatorial types yields a graded partially ordered set $P_g$ on the cells of $\Mtrg$. The rank of a combinatorial type $(G,w)$ is $|E(G)|$.
\end{lemma}

\begin{proof}
  It is clear that we obtain a poset; furthermore, $(G',w')$ is covered by $(G,w)$ precisely if $(G',w')$ is obtained from $(G,w)$ by contracting a single edge. The formula for rank then follows.
\end{proof}

For example, $P_2$ is shown in  Figure \ref{f:m_2}; it also appeared in
 \cite[Figure 1]{bmv}.  The poset $P_3$ is shown in Figure \ref{f:m_3}.  It is color-coded according to the Torelli map, as explained in Section 6.  

\begin{figure}
\centering
\mbox{\subfigure{\includegraphics[width=2in]{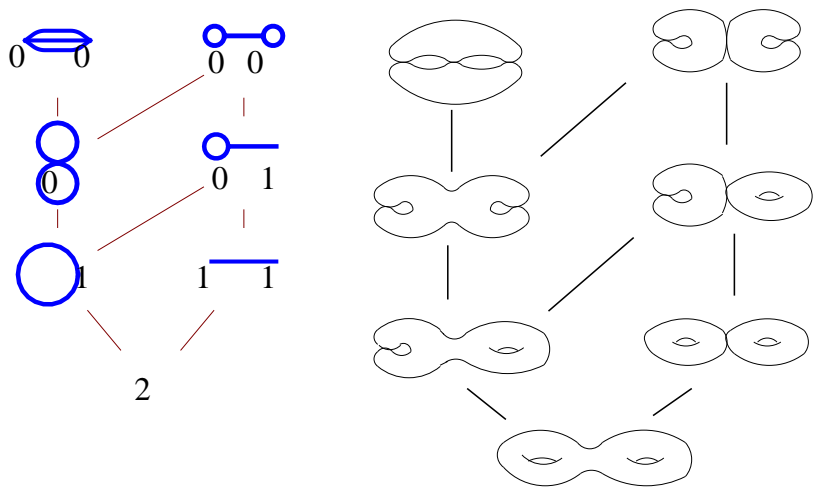}}\quad
\subfigure{\includegraphics[height=2.5in]{m_2classical.pstex} }}
\vskip -0.1in
\caption{Posets of cells of $M^{\text{tr}}_2$ (left) and of $\overline{\mathcal{M}_2}$ (right).} 
\label{f:m_2}
\end{figure}

Our goal is to compute $P_g$. We do so by first listing its maximal elements, and then computing all possible specializations of those combinatorial types.
For the first step, we use Proposition 3.2.4(i) in \cite{bmv}, which characterizes the maximal cells of $\Mtrg$: they correspond precisely to combinatorial types $(G,\overline{0})$, where $G$ is a connected $3$-regular graph of genus $g$, and $\overline{0}$ is the zero weight function on $V(G)$.
Connected, $3$-regular graphs of genus $g$ are equivalently characterized as connected, $3$-regular graphs on $2g-2$ vertices. These have been enumerated:

\begin{proposition}
  The number of maximal cells of $\Mtrg$ is equal to the $(g-1)^{st}$ term in the sequence
  \[ 2, 5, 17, 71, 388, 2592, 21096, 204638, 2317172, 30024276, 437469859, \dots \]
\end{proposition}
\begin{proof}
  This is sequence A005967 in \cite{sloane}, whose $g^{th}$ term is the number of connected $3$-regular graphs on $2g$ vertices.
\end{proof}

In fact, the connected, $3$-regular graphs of genus $g$ have been conveniently written down for $g$ at most 
$6$. This work was done in the 1970s by Balaban, a chemist whose interests along these lines were in molecular applications of graph theory. The graphs for $g \leq 5$ appear in his article \cite{b}, and the $388$ genus $6$ graphs appear in \cite{b2}.

Given the maximal cells of $\Mtrg$, we can compute the rest of them:

\begin{algorithm} \ \\
  Input: Maximal cells of $\Mtrg$ \\
  Output: Poset of all cells of $\Mtrg$
  \begin{enumerate}
    \item[1.] Initialize $P_g$ to be the set of all maximal cells of $\Mtrg$, with no relations. Let $L$ be a list of elements of $P_g$.
    \item[2.] While $L$ is nonempty:
    \begin{itemize}
      \item[] Let $(G,w)$ be the first element of $L$. Remove $(G,w)$ from $L$. Compute all 1-edge contractions of $(G,w)$. \\
          For each such contraction $(G',w')$:
          \begin{itemize}
            \item[] If $(G',w')$ is isomorphic to an element $(G'',w'')$ already in the poset $P_g$, add a cover relation $(G'',w'') \leq (G,w)$. \\
                Else, add $(G',w')$ to $P_g$ and add a cover relation $(G',w') \leq (G,w)$. Add $(G',w')$ to the list $L$.
          \end{itemize}
    \end{itemize}
    \item[3.] Return $P_g$.
  \end{enumerate}
\end{algorithm}

We implemented this algorithm in \textsc{Mathematica}. The most costly step is computing graph isomorphisms in Step 2. 
Our results are summarized in the following theorem. By an $f$-vector of a poset, we mean the vector whose $i$-th entry is the number of elements of rank~$i-1$.

\begin{theorem} We obtained the following computational results:
\label{t:comp}{\ } \begin{enumerate}
	\item 
  The moduli space $M^{tr}_3$ has 42 cells and $f$-vector
  \[ (1,2,5,9,12,8,5). \]
  Its poset of cells $P_3$ is shown in Figure \ref{f:m_3}.
  \item
  The moduli space $M^{tr}_4$ has 379 cells and $f$-vector
  \[ (1,3,7,21,43,75,89,81,42,17). \]
  \item
   The moduli space $M^{tr}_5$ has 4555 cells and $f$-vector
\[ (1,3,11,34,100,239,492,784,1002,926,632,260,71).\]
\end{enumerate}
\end{theorem}

The posets $P_4$ and $P_5$ are much too large to display here, but are available at
the website above.

\begin{remark}
The data of $P_3$, illustrated in Figure \ref{f:m_3}, is related, but not identical, to the data obtained by T.~Brady in \cite[Appendix A]{br}.  In that paper, the author enumerates the cells of a certain deformation retract, called $K_3$, of Culler-Vogtmann Outer Space \cite{cullervogtmann}, modulo the action of the group $\Out(F_3)$.  In that setting, one only needs to consider bridgeless graphs with all vertices of weight 0, thus throwing out all but 8 cells of the poset $P_3$.  In turn, the cells of $K_3/\Out(F_n)$ correspond to chains in the poset on those eight cells.  It is these chains that are listed in Appendix A of \cite{br}.  We believe that further exploration of the connection between Outer Space and $M^{\text{tr}}_g$ would be interesting to researchers in
both tropical geometry and geometric group theory.
\end{remark}

\begin{remark}
What is the topology of $\Mtrg$?  Of course, $\Mtrg$ is always contractible: there is a deformation retract onto the unique 0-dimensional cell.  So to make this question interesting, we restrict our attention to the subspace $M^{tr'}_g$ of $\Mtrg$ consisting of graphs with total edge length 1, say.  For example, by looking at Figure \ref{f:m2_fan}, we can see that $M^{tr'}_2$ is still contractible.  We would like to know if the space $M^{tr'}_g$ is also contractible for larger $g$.
\end{remark}
 
\section{Stacky fans}
In Section 2, we defined the space $\Mtrg$. In Sections 4 and 6, we will define the space $\Atrg$ and the Torelli map $\ttrg: \Mtrg \rightarrow \Atrg$. For now, however, let us pause and define the category of stacky fans, of which $\Mtrg$ and $\Atrg$ are objects and $\ttrg$ is a morphism. The reader is invited to keep $\Mtrg$ in mind as a running example of a stacky fan.

The purpose of this section is to offer a new definition of stacky fan, Definition \ref{d:sf}, which we hope fixes an inconsistency in the definition by Brannetti, Melo, and Viviani, in Section 2.1 of \cite{bmv}.  We believe that their condition for integral-linear gluing maps is too restrictive and fails for $\Mtrg$ and $\Atrg$.  However, we do think that their definition of a stacky fan morphism is correct, so we repeat it 
in Definition \ref{d:sfm}.  
We also prove that $\Mtrg$ is a stacky fan according to our new definition.  The proof for $\Atrg$ is deferred to \S 4.3.

\begin{definition}
A {\bf rational open polyhedral cone} in $\R^n$ is a subset of $\R^n$ of the form
$\, \{a_1x_1 + \cdots + a_tx_t : a_i \in \R_{>0} \}$,
for some  fixed vectors $x_1, \dots, x_t \in \Z^n$.  By convention, we also allow the trivial cone $\{0\}$.
\end{definition}

\begin{definition}
\label{d:sf}
  Let $X_1 \subseteq \R^{m_1}, \dots, X_k \subseteq \R^{m_k}$ be full-dimensional rational open polyhedral cones. For each $i = 1, \dots, k$, let $G_i$ be the subgroup of $GL_{m_i}(\Z)$ which fixes the cone $X_i$ setwise, and let $X_i/G_i$ denote the topological quotient thus obtained.  The action of $G_i$ on $X_i$ extends naturally to an action of $G_i$ on the Euclidean closure $\overline{X_i}$, and we let $\overline{X_i}/G_i$ denote the quotient.

Suppose that we have a topological space $X$ and, for each $i = 1, \dots, k$, a continuous map
\[ \alpha_i: \frac{\overline{X_i}}{G_i} \rightarrow X. \]
Write $C_i = \alpha_i\left(\frac{X_i}{G_i}\right)$ and $\overline{C_i} = \alpha_i\left(\frac{\overline{X_i}}{G_i}\right)$ for each $i$.  Given $Y\subseteq X_i$, we will abuse notation by writing $\alpha_i(Y)$ for $\alpha_i$ applied to the image of $Y$ under the map $\overline{X_i}\twoheadrightarrow \overline{X_i}/G_i$. 

Suppose that the following properties hold for each index $i$:
\begin{enumerate}
  \item The restriction of $\alpha_i$ to $ \frac{X_i}{G_i} $ is a homeomorphism onto $C_i$,
  \item We have an equality of sets $X = \coprod C_i$,
  \item For each cone $\overline{X_i}$ and for each face $F_i$ of $\overline{X_i}$, $\alpha_i(F_i) = \overline{C_l}$ for some $l$. Furthermore, $\dim F_i = \dim \overline{X_l} = m_l$, and there is an $\R$-invertible linear map $L: \sspan \langle F_i \rangle \cong \R^{m_l} \rightarrow \R^{m_l}$ such that 

\begin{itemize}
 \item $L(F_i) = \overline{X_l}$,
\item $L(\Z^{m_i} \cap \sspan(F_i)) = \Z^{m_l}$, and
\item the following diagram commutes:
      \[ \xymatrix@H=0.1in{
      F_i \ar[dd]_L \ar[rrd]^{\alpha_i} & & \\
      & & \overline{C_l} \\
      \overline{X_l} \ar[rru]_{\alpha_l}& &} \]
\end{itemize}
We say that $\overline{C_l}$ is a {\bf stacky face} of $\overline{C_i}$ in this situation.
  \item For each pair $i,j$,
  \[ \overline{C_i} \cap \overline{C_j} = C_{l_1} \cup \dots \cup C_{l_t},\]
   where $C_{l_1}, \dots, C_{l_t}$ are the common stacky faces of $\overline{C_i}$ and $\overline{C_j}$.
\end{enumerate}
Then we say that $X$ is a {\bf stacky fan}, with cells $\{X_i/G_i\}$.
\end{definition}

\begin{remark}
Condition (iii) in the definition above essentially says that $\overline{X_i}$ has a face $F_i$ that looks ``exactly like'' $\overline{X_l}$, even taking into account where the lattice points are.  It plays the role of the usual condition on polyhedral fans that the set of cones is closed under taking faces.  Condition (iv) replaces the usual condition on polyhedral fans that the intersection of two cones is a face of each. Here, we instead allow unions of common faces.
\end{remark}

\begin{theorem}
  The moduli space $\Mtrg$ is a stacky fan with cells
  \[ {C(G,w)} = \frac{\R^{E(G)}_{>0}}{\Aut(G,w)} \]
  as $(G,w)$ ranges over genus $g$ combinatorial types. Its points are in bijection with tropical curves of genus $g$.
\end{theorem}

\begin{proof}
  Recall that
  \[ \Mtrg = \frac{\coprod \overline{C(G,w)}}{\sim}, \]
  where $\sim$ is the relation generated by contracting zero-length edges. Thus, each equivalence class has a unique representative $(G_0,w,l)$ corresponding to an honest metric graph: one with all edge lengths positive. This gives the desired bijection.

  Now we prove that $\Mtrg$ is a stacky fan. For each $(G,w)$, let
  \[ \alpha_{G,w}: \overline{C(G,w)} \rightarrow \frac{\coprod\overline{C(G',w')}}{\sim} \]
  be the natural map. Now we check each of the requirements to be a stacky fan, in the order (ii), (iii), (iv), (i).

  For (ii), the fact that
  \[ \Mtrg = \coprod C(G,w) \]
  follows immediately from the observation above.  
  
  Let us prove (iii).  Given a combinatorial type $(G,w)$, the corresponding closed cone is $\R^{E(G)}_{\geq 0}$. A face $F$ of $\R^{E(G)}_{\geq 0}$ corresponds to setting edge lengths of some subset $S$ of the edges to zero. Let $(G',w')$ be the resulting combinatorial type, and let $\pi: E(G) \setminus S \rightarrow E(G')$ be the natural bijection (it is well-defined up to $(G',w')$-automorphisms, but this is enough). Then $\pi$ induces an invertible linear map
  \[ L_\pi: \R^{E(G)\setminus S} \longrightarrow \R^{E(G')} \]
  with the desired properties. Note also that the stacky faces of $\overline{C(G,w)}$ are thus all possible specializations $\overline{C(G',w')}$.

  For (iv), given two combinatorial types $(G,w)$ and $(G',w')$, then
  \[ \overline{C(G,w)} \cap \overline{C(G',w')} \]
  consists of the union of all cells corresponding to common specializations of $(G,w)$ and $(G',w')$. As noted above, these are precisely the common stacky faces of $\overline{C(G,w)}$ and $\overline{C(G',w')}$.

 For (i), we show that $\alpha_{G,w}$ restricted to
$C(G,w) = \R^{E(G)}_{>0}/\Aut(G,w)$
  is a homeomorphism onto its image. It is continuous by definition of $\alpha_{G,w}$ and injective by definition of $\sim$. Let $V$ be closed in $C(G,w)$, say $V = W \cap C(G,w)$ where $W$ is closed in $\overline{C(G,w)}$. To show that $\alpha_{G,w}(V)$ is closed in $\alpha_{G,w}(C(G,w))$,
  it suffices to show that $\alpha_{G,w}(W)$ is closed in $\Mtrg$. Indeed,
  the fact that the cells $C(G,w)$ are pairwise disjoint in $\Mtrg$ implies that
  \[ \alpha_{G,w}(V) = \alpha_{G,w}(W) \cap \alpha_{G,w}(C(G,w)). \]
 Now, note that $\Mtrg$ can equivalently be given as the quotient of the space $$\coprod_{(G,w)} \R^{E(G)}_{\ge 0}$$ by all possible linear maps $L_\pi$ arising as in the proof of (iii).  All of the maps $L_\pi$ identify faces of cones with other cones.  Now let $\tilde{W}$ denote the lift of $W$ to $ \R^{E(G)}_{\ge 0}$; then for any other type $(G',w')$, we see that the set of points in $ \R^{E(G')}_{\ge 0}$ that are identified with some point in  $\tilde{W}$ is both closed and $\Aut(G',w')$-invariant, and passing to the quotient $ \R^{E(G')}_{\ge 0}/\Aut(G',w')$ gives the claim.
\end{proof}

We close this section with the definition of a morphism of stacky fans.  The tropical Torelli map, which we will define in Section~6, will be an~example.

\begin{definition}\cite[Definition 2.1.2]{bmv}\label{d:sfm}
  Let $$X_1 \subseteq \R^{m_1}, \dots, X_k \subseteq \R^{m_k},
  \,\,Y_1 \subseteq \R^{n_1},\dots,Y_l \subseteq \R^{n_l}$$ be full-dimensional rational open polyhedral cones.
  Let $G_1 \subseteq GL_{m_1}(\Z),$ $\dots, G_k \subseteq GL_{m_k}(\Z),H_1 \subseteq GL_{n_1}(\Z), \dots, H_l \subseteq GL_{n_l}(\Z)$ be groups stabilizing $X_1$, $\dots$, $X_k$, $Y_1, \dots, Y_l$ respectively. Let $X$ and $Y$ be stacky fans with cells 
$$\left\{\frac{X_i}{G_i}\right\}^k_{i=1},\left\{\frac{Y_j}{H_j}\right\}^l_{j=1},$$
Denote by $\alpha_i$ and $\beta_j$ the maps $\frac{\overline{X_i}}{G_i} \rightarrow X$ and $\frac{\overline{Y_j}}{H_j} \rightarrow Y$ that are part of the stacky fan data of $X$ and $Y$.

  Then a {\bf morphism of stacky fans} from $X$ to $Y$ is a continuous map $\pi: X \rightarrow Y$ such that for each cell $X_i / G_i$ there exists a cell $Y_j/H_j$ such that
  \begin{enumerate}
    \item $\pi\left(\alpha_i\left(\frac{X_i}{G_i}\right)\right) \subseteq \beta_j \left(\frac{Y_j}{H_j}\right)$, and
    \item there exists an integral-linear map
    \[ L: \R^{m_i} \rightarrow \R^{n_j}, \]
    restricting to a map
    \[ L: X_i \rightarrow Y_j, \]
    such that the diagram below commutes:
    \[ \xymatrix{
X_i \ar[rr]\ar[d]_L & &\alpha_i(X_i/G_i)\ar[d]^\pi \\
Y_j \ar[rr] & &\beta_j(Y_j/H_j)
} \]
  \end{enumerate}
\end{definition}

\section{Principally polarized tropical abelian varieties}
The purpose of this section is to construct the moduli space of principally polarized tropical abelian varieties, denoted $\Atrg$. Our construction is different from the one in \cite{bmv}, though it is still very much inspired by the ideas in that paper. The reason for presenting a new construction here is that a topological subtlety in the construction there prevents their space from being a stacky fan as claimed in \cite[Thm.~4.2.4]{bmv}.

We begin in \S 4.1 by recalling the definition of a principally  polarized tropical abelian variety. In 
\S 4.2, we review the theory of Delone subdivisions and the main theorem of Voronoi reduction theory. We construct $\Atrg$ in \S 4.3 and prove that it is a stacky fan and that it is Hausdorff.  We remark on the difference between our construction and the one in \cite{bmv} in \S 4.4.

\subsection{Definition of principally polarized tropical abelian variety} Fix $g \geq 1$. Following \cite{bmv}, we define a {\bf principally polarized tropical abelian variety}, or {\bf pptav} for short, to be a pair
\[ (\R^g/ \Lambda, Q), \]
where $\Lambda$ is a lattice of rank $g$ in $\R^g$ (i.e. a discrete subgroup of $\R^g$ that is isomorphic to $\Z^g$), and $Q$ is a positive semidefinite quadratic form on $\R^g$ whose nullspace is rational with respect to $\Lambda$. By this, we mean that the subspace $\ker(Q) \subseteq \R^g$ has a vector space basis whose elements are each of the form
\[ a_1 \lambda_1 + \cdots + a_k\lambda_k,\quad a_i \in \Q, \lambda_i \in \Lambda. \]
We say that $Q$ has {\bf rational nullspace} if its nullspace is rational with respect to $\Z^g$.

We say that two pptavs $(\R^g/\Lambda,Q)$ and $(\R^g/\Lambda',Q')$ are isomorphic if there exists a matrix $X \in GL_g(\R)$ such that
\begin{itemize}
  \item left multiplication by $X^{-1}$ sends $\Lambda$ isomorphically to $\Lambda'$, that is, the map $X^{-1}:\R^g \to \R^g$ sending a column vector $v$ to $X^{-1}v$ restricts to an isomorphism of lattices $\Lambda$ and $\Lambda'$; and
  \item $Q' = X^T Q X$.
\end{itemize} 

Note that any pptav $(\R^g/\Lambda,Q)$ is isomorphic to one of the form $(\R^g/\Z^g,Q')$, namely by taking $X$ to be any matrix sending $\Z^g$ to $\Lambda$ and setting $Q' = X^TQX$. Furthermore, $(\R^g/\Z^g,Q)$ and $(\R^g/\Z^g,Q')$ are isomorphic if and only if  there exists $X \in GL_g(\Z)$ with  $X^TQX = Q'$.

We are interested in pptavs only up to isomorphism.  Therefore, we might be tempted to define the moduli space of pptavs to be the quotient of the topological space $\widetilde{S}^g_{\geq 0}$, the space of positive semidefinite matrices with rational nullspace, by the action of $GL_g(\Z)$. But, as we will see in Section 4.4., this quotient space 
is unexpectedly thorny: for $g \geq 2$, it is not even Hausdorff! 

It turns out that we can solve this problem by first grouping matrices together into cells according to their Delone subdivisions, and then gluing the cells together to obtain the full moduli space. We review the theory of Delone subdivisions next.

\subsection{Voronoi reduction theory}
Recall that a matrix has rational nullspace if its kernel has a basis defined over $\Q$.

\begin{definition}
  Let $\widetilde{S}^g_{\geq 0}$ denote the set of $g \times g$ positive semidefinite matrices with rational nullspace. By regarding a $g \times g$ symmetric real matrix as a vector in $\R^{\binom{g+1}{2}}$, with one coordinate for each diagonal and above-diagonal entry of the matrix, we view $\widetilde{S}^g_{\geq 0}$ as a subset of $\R^{\binom{g+1}{2}}$.
\end{definition}

The group $GL_g(\Z)$ acts on $\widetilde{S}^g_{\geq 0}$ on the right by changing basis:
\[ Q \cdot X = X^TQX, \text{ for all } X \in GL_g(\Z), Q \in \widetilde{S}^g_{\geq 0}. \]

\begin{definition}
  Given $Q \in \widetilde{S}^g_{\geq 0}$, define $\Del(Q)$ as follows. Consider the map $l : \Z^g \to \Z^g \times \R$ sending $x \in \Z^g$ to $(x,x^TQx)$. View the image of $l$ as an infinite set of points in $\R^{g+1}$, one above each point in $\Z^g$, and consider the convex hull of these points. The lower faces of the convex hull (the faces that are visible from $(0,-\infty)$) can now be projected to $\R^g$ by the map $\pi:\R^{g+1}\to \R^g$ that forgets the last coordinate. This produces an infinite periodic polyhedral subdivision of $\R^g$, called the {\bf Delone subdivision} of $Q$ and denoted $\Del(Q)$.
\end{definition}

Now, we group together matrices in $\widetilde{S}^g_{\geq 0}$ according to the Delone subdivisions to which they correspond.

\begin{definition}
  Given a Delone subdivision $D$, let
  \[ \sigma_D = \{ Q \in \widetilde{S}^g_{\geq 0} : \Del(Q) = D \}. \]
\end{definition}

\begin{proposition}\cite{voronoi1908}
The set $\sigma_D$ is an open rational polyhedral cone in $\widetilde{S}^g_{\geq 0}$.
\end{proposition}

\begin{figure}%
\includegraphics[height=3in]{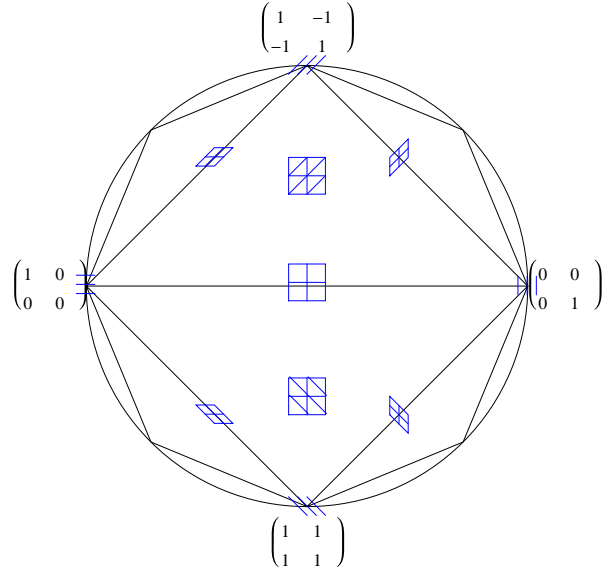}%
\caption{Infinite decomposition of $\widetilde{S}^2_{\geq 0}$ into secondary cones.}%
\label{f:s2}%
\end{figure}

Let $\overline{\sigma_D}$ denote the Euclidean closure of $\sigma_D$ in $\R^{\binom{g+1}{2}}$, so $\overline{\sigma_D}$ is a closed rational polyhedral cone. We call it the {\bf secondary cone} of $D$.

\begin{example}
Figure \ref{f:s2} shows the decomposition of $\widetilde{S}^2_{\geq 0}$ into secondary cones.  Here is how to interpret the picture.  First, points in $\widetilde{S}^2_{\geq 0}$ are $2\times 2$ real symmetric matrices, so let us regard them as points in $\R^3$.  Then $\widetilde{S}^2_{\geq 0}$ is a cone in $\R^3$.  Instead of drawing the cone in $\R^3$, however, we only draw a hyperplane slice of it.  Since it was a cone, our drawing does not lose information.  For example, what looks like a point in the picture, labeled by the matrix {\small \mm{1}{0}{0}{0}}, really is the ray in $\R^3$ passing through the point $(1,0,0)$.
\end{example}

 Now, the action of the group $GL_g(\Z)$ on $\widetilde{S}^g_{\geq 0}$ extends naturally to an action (say, on the right) on subsets of $\widetilde{S}^g_{\geq 0}$.  In fact, given $X \in GL_g(\Z)$ and $D$ a Delone subdivision, 
\[ \sigma_D \cdot X = \sigma_{X^{-1}D} \quad\text{  and  }\quad \overline{\sigma_D} \cdot X = \overline{\sigma_{X^{-1}D}}. \]
So $GL_g(\Z)$ acts on the set $$\{\overline{\sigma_D}:D \text{ is a Delone subdivision of } \R^g\}.$$ Furthermore, $GL_g(\Z)$ acts on the set of Delone subdivisions, with action induced by 
the action of $GL_g(\Z)$ on $\R^g$.
Two cones $\sigma_D$ and $\sigma_{D'}$ are $GL_g(\Z)$-equivalent iff $D$ and $D'$ are.

\begin{theorem}[Main theorem of Voronoi reduction theory \cite{voronoi1908}]
\label{t:main}
The set of secondary cones
\[ \{ \overline{\sigma_D} : D \text{ is a Delone subdivision of } \R^g \} \]
yields an infinite polyhedral fan whose support is $\widetilde{S}^g_{\geq 0}$, known as the {\bf second Voronoi decomposition}. There are only finitely many $GL_g(\Z)$-orbits of this set.
\end{theorem}

\subsection{Construction of $\Atrg$}
Equipped with Theorem \ref{t:main}, we will now construct our tropical moduli space $\Atrg$. We will show that its points are in bijection with the points of $\widetilde{S}^g_{\geq 0} / GL_g(\Z)$, and that it is a stacky fan whose cells correspond to $GL_g(\Z)$-equivalence classes of Delone subdivisions of $\R^g$.

\begin{definition}
  Given a Delone subdivision $D$ of $\R^g$, let
  \[ \Stab(\sigma_D) = \{X \in GL_g(\Z) : \sigma_D \cdot X = \sigma_D \} \]
  be the setwise stabilizer of $\sigma_D$. 
\end{definition}

Now,  the subgroup $\Stab(\sigma_D) \subseteq GL_g(\Z)$ acts on the 
open cone $\sigma_D$, and we may extend this action to an action on its closure $\overline{\sigma_D}$.

\begin{definition}
  Given a Delone subdivision $D$ of $\R^g$, let
  \[ C(D) \,\,= \,\,\overline{\sigma_D}/\Stab(\sigma_D). \]
  Thus, $C(D)$ is the topological space obtained as a quotient of the rational polyhedral cone $\overline{\sigma_D}$ by a group action.
\end{definition}

Now, by Theorem \ref{t:main}, there are only finitely many $GL_g(\Z)$-orbits of secondary cones $\overline{\sigma_D}$. Thus, we may choose $D_1, \dots, D_k$ Delone subdivisions of $\R^g$ such that $\overline{\sigma_{D_1}}, \dots, \overline{\sigma_{D_k}}$ are representatives for $GL_g(\Z)$-equivalence classes of secondary cones. (Note that we do not need anything like the Axiom of Choice to select these representatives. Rather, we can use Algorithm 1 in \cite{v}.  We start with a particular Delone triangulation and then walk across codimension 1 faces to all of the other ones; then we compute the faces of these maximal cones to obtain the nonmaximal ones.  The key idea that allows the algorithm to terminate is that all maximal cones are related to each other by finite sequences of ``bistellar flips'' as described in Section 2.4 of \cite{v}).

\begin{definition} \label{d:ag}
Let $D_1, \dots, D_k$ be Delone subdivisions such that $\overline{\sigma_{D_1}}$, $\dots, \overline{\sigma_{D_k}}$ are representatives for $GL_g(\Z)$-equivalence classes of secondary cones in $\R^g$. Consider the disjoint union
\[ C(D_1) \coprod \cdots \coprod C(D_k), \]
and define an equivalence relation $\sim$ on it as follows. Given $Q_i \in \overline{\sigma(D_i)}$ and $Q_j \in \overline{\sigma(D_j)}$, let $[Q_i]$ and $[Q_j]$ be the corresponding elements in $C(D_i)$ and $C(D_j)$, respectively. Now let
\[ [Q_i] \sim [Q_j] \]
if and only if $Q_i$ and $Q_j$ are $GL_g(\Z)$-equivalent matrices in $\widetilde{S}^g_{\geq 0}$. Since $\Stab(\sigma_{D_i})$, $\Stab(\sigma_{D_j})$ are subgroups of $GL_g(\Z)$, the relation $\sim$ is defined independently of choice of representatives $Q_i$ and $Q_j$, and is clearly an equivalence relation.

We now define the {\bf moduli space of principally polarized tropical abelian varieties}, denoted $\Atrg$, to be the topological space
\[ \Atrg \,\,\,= \,\,\, \coprod\limits^k_{i=1} C(D_k) / \sim. \]
\end{definition}

\begin{example}
\label{e:Atr2}
  Let us compute $A^{\text{tr}}_2$. Combining the taxonomies in Sections 4.1 and 4.2 of \cite{v}, we may choose four representatives $D_1, D_2, D_3, D_4$ for orbits of secondary cones as in Figure \ref{f:subdiv}.

\begin{figure}[h]%
\includegraphics[width=3in]{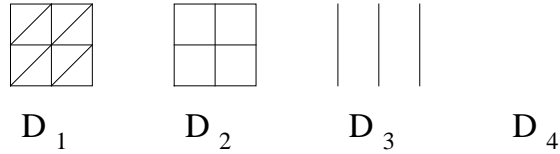}%
\caption{Cells of $A^{\text{tr}}_2$.}%
\label{f:subdiv}%
\end{figure}
We can describe the corresponding secondary cones as follows: let
$R_{12} = \left(\begin{matrix}1&-1\\-1&1\end{matrix}\right),
R_{13} = \left(\begin{matrix}1&0\\0&0\end{matrix}\right),
R_{23} = \left(\begin{matrix}0&0\\0&1\end{matrix}\right).$
Then
\begin{align*}
 \overline{\sigma_{D_1}} &= \R_{\geq 0} \langle R_{12}, R_{13}, R_{23} \rangle, \\
 \overline{\sigma_{D_2}} &= \R_{\geq 0} \langle R_{13}, R_{23} \rangle, \\
 \overline{\sigma_{D_3}} &= \R_{\geq 0} \langle R_{13} \rangle, \text{ and } \\
 \overline{\sigma_{D_4}} &= \{0\}.
\end{align*}

\begin{figure}%
\includegraphics[width=2in]{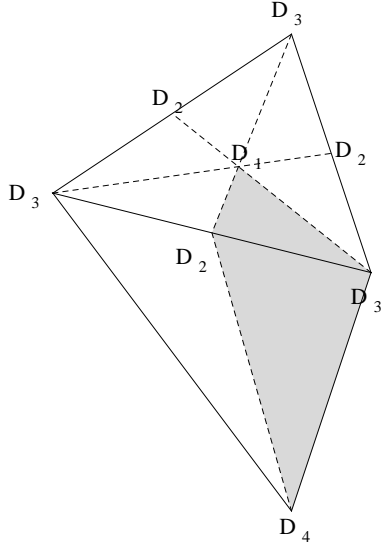}%
\caption{The stacky fan $A^{\text{tr}}_2$.  The shaded area represents a choice of fundamental domain.}%
\label{f:a2_fan}%
\end{figure}

Note that each closed cone $\overline{\sigma_{D_2}}, \overline{\sigma_{D_3}}, \overline{\sigma_{D_4}}$ is just a face of $\overline{\sigma_{D_1}}$. One may check -- and we will, in Section 5 -- that for each $j = 2,3,4,$ two matrices $Q,Q'$ in $\overline{\sigma_{D_j}}$ are $\Stab(\sigma_{D_j})$-equivalent if and only if they are $\Stab(\sigma_{D_1})$-equivalent. Thus, gluing the cones $C(D_2)$, $C(D_3)$, and $C(D_4)$ to $C(D_1)$ does not change $C(D_1)$. We will see in Theorem \ref{t:aut} that the action of $\Stab(\sigma_{D_1})$ on $\overline{\sigma_{D_1}}$ is an $S_3$-action that permutes the three rays of $\overline{\sigma_{D_1}}$. So we may pick a fundamental domain, say the closed cone
\[ C = \R_{\geq 0} \left\langle
\left(\begin{matrix} 0 & 0 \\ 0 & 1 \end{matrix}\right),
\left(\begin{matrix} 1 & 0 \\ 0 & 1 \end{matrix}\right),
\left(\begin{matrix} 2 & -1 \\ -1 & 2 \end{matrix}\right)
\right\rangle,
\]
and conclude that $C(D_1)$, and hence $A^{\text{tr}}_2$, is homeomorphic to $C$. See Figure \ref{f:a2_fan} for a picture of $A^{\text{tr}}_2$.  Of course, $A^{\text{tr}}_2$ has further structure, as the next theorem shows.
\end{example}

\begin{theorem}\label{t:agsf}
  The space $\Atrg$ constructed in Definition \ref{d:ag} is a stacky fan with cells
    $\,  \sigma_{D_i}/\Stab(\sigma_{D_i})\,$ for $\,i=1,\ldots,k$.
\end{theorem}

\begin{proof}
  For each $i = 1, \dots, k$, let $\alpha_i$ be the composition
  \[ \frac{\overline{\sigma_{D_i}}}{\Stab(\sigma_{D_i})}
  \,\, \overset{\gamma_i}{\hooklongrightarrow} \,\,
  \coprod_{j=1}^k C(D_j) \,\,\overset{q}{\longrightarrow}\,\, \left(\coprod\limits_{j=1}^k C(D_j)\right)/{\sim}, \]
  where $\gamma_i$ is the inclusion of $C(D_i) = \frac{\overline{\sigma_{D_i}}}{\Stab(\sigma_{D_i})}$ into $\coprod_{j=1}^k C(D_j)$ and $q$ is the quotient map. Now we check the four conditions listed in Definition \ref{d:sf} for $\Atrg$ to be a stacky fan. 

  First, we prove that the restriction of $\alpha_i$ to $\frac{\sigma_{D_i}}{\Stab(\sigma_{D_i})}$ is a homeomorphism onto its image. Now, $\alpha_i$ is continuous since both $\gamma_i$ and $q$ are. To show that $\alpha_i|_{\frac{\sigma_{D_i}}{\Stab(\sigma_{D_i})}}$ is one-to-one onto its image, let $Q,Q' \in \sigma_{D_i}$ such that $\alpha_i([Q])=\alpha_i([Q'])$. Then $[Q] \sim [Q']$, so there exists $A \in GL_g(\Z)$ such that $Q'=A^TQA$. Hence $Q' \in A^T\sigma_{D_i}A = \sigma_{A^{-1}D_i}$. Thus $\sigma_{A^{-1}D_i}$ and $\sigma_{D_i}$ intersect, hence $\sigma_{A^{-1}D_i} = \sigma_{D_i}$ and $A \in \Stab(\sigma_{D_i})$. So $[Q] = [Q']$.

  Thus, $\alpha_i|_{\frac{\sigma_{D_i}}{\Stab(\sigma_{D_i})}}$ has a well-defined inverse map, and we wish to show that this inverse map is continuous. Let $X \subseteq \frac{\sigma_{D_i}}{\Stab \sigma_{D_i}}$ be closed; we wish to show that $\alpha_i(X)$ is closed in $\alpha_i\left(\frac{\sigma_{D_i}}{\Stab \sigma_{D_i}}\right)$. Write $X = Y \cap \frac{\sigma_{D_i}}{\Stab \sigma_{D_i}}$ where $Y \subseteq \frac{\overline{\sigma_{D_i}}}{\Stab \sigma_{D_i}}$ is closed.  Then
  \[ \alpha_i(X) = \alpha_i(Y) \cap \alpha_i \left(\frac{\sigma_{D_i}}{\Stab \sigma_{D_i}}\right); \]
  this follows from the fact that  $GL_g(\Z)$-equivalence never identifies a point on the boundary of a closed cone with a point in the relative interior.
  So we need only show that $\alpha_i(Y)$ is closed in $\Atrg$.  To be clear: we want to show that given any closed $Y \subseteq \frac{\overline{\sigma_{D_i}}}{\Stab \sigma_{D_i}}$, the image $\alpha_i(Y) \subseteq \Atrg$ is closed.

  Let $\tilde{Y} \subseteq \overline{\sigma_{D_i}}$ be the preimage of $Y$ under the quotient map
  \[ \overline{\sigma_{D_i}} \twoheadlongrightarrow \frac{\overline{\sigma_{D_i}}}{\Stab \sigma_{D_i}}. \]
  Then, for each $j = 1, \dots, k$, let
  \[ \tilde{Y_j} = \{Q \in \overline{\sigma_{D_j}} : Q \equiv_{GL_g(\Z)} Q' \text{ for some } Q' \in \tilde{Y} \} \subseteq \overline{\sigma_{D_j}}. \]
  We claim each $\tilde{Y_j}$ is closed in $\overline{\sigma_{D_j}}$. First, notice that for any $A \in GL_g(\Z)$, the cone $A^T\overline{\sigma_{D_i}}A$ intersects $\overline{\sigma_{D_j}}$ in a (closed) face of $\overline{\sigma_{D_j}}$ (after all, the cones form a polyhedral subdivision).  In other words, $A$ defines an integral-linear isomorphism $L_A: F_{A,i}\rightarrow F_{A,j}$ sending $X\mapsto A^TXA$, where $F_{A,i}$ is a face of $\overline{\sigma_{D_i}}$ and $F_{A,j}$ is a face of $\overline{\sigma_{D_j}}$.  Moreover,
the map $L_A$ is entirely determined by three choices: the choice of $ F_{A,i}$, the choice of $ F_{A,j}$, and the choice of a bijection between the rays of $F_{A,j}$ and $F_{A,j}$.  Thus there exist only finitely many distinct such maps.  Therefore 
\[
 \tilde{Y_j} = \bigcup_{A\in GL_g(\Z)} L_A(\tilde{Y} \cap F_{A,i})= \bigcup_{k=1}^s L_{A_k}(\tilde{Y} \cap F_{{A_k},i})
\]
for some choice of finitely many matrices $A_1, \dots, A_s \in GL_g(\Z)$.  Now, each $L_A$ is a homeomorphism, so each $L_A(\tilde{Y} \cap F_{A,i})$ is closed in $F_{A,j}$ and hence in $\overline{\sigma_{D_j}}$.  So $\tilde{Y_j}$ is closed.

  Finally, let $Y_j$ be the image of $\tilde{Y_j} \subseteq \overline{\sigma_{D_j}}$ under the quotient map
  \[ \overline{\sigma_{D_j}} \overset{\pi_i}{\twoheadlongrightarrow}  \frac{\overline{\sigma_{D_j}}}{\Stab \sigma_{D_j}}. \]
  Since $\pi_j^{-1}(Y_j) = \tilde{Y_j}$, we have that $Y_j$ is closed.
  Then the inverse image of $\alpha_i(Y)$ under the quotient map
  \[ \coprod_{j=1}^k C(D_j) \longrightarrow \left(\coprod\limits_{j=1}^k C(D_j)\right)/{\sim} \]
  is precisely $Y_1 \coprod \cdots \coprod Y_k$, which is closed. Hence $\alpha_i(Y)$ is closed. This finishes the proof that $\alpha_i|_{\frac{\sigma_{D_i}}{\Stab(\sigma_{D_i})}}$ is a homeomorphism onto its image.

  Property (ii) of being a stacky fan follows from the fact that any matrix $Q \in \widetilde{S}^g_{\geq 0}$ is $GL_g(\Z)$-equivalent only to some matrices in a single chosen cone, say $\sigma_{D_i}$, and no others. Here, $\Del(Q)$ and $D_i$ are $GL_g(\Z)$-equivalent. Thus, given a point in $\Atrg$ represented by $Q \in \widetilde{S}^g_{\geq 0}$, $Q$ lies in $\alpha_i \left(\frac{\sigma_{D_i}}{\Stab \sigma_{D_i}}\right)$ and no other $\alpha_j\left(\frac{\sigma_{D_j}}{\Stab \sigma_{D_j}}\right)$, and is the image of a single point in $\frac{\sigma_{D_i}}{\Stab \sigma_{D_i}}$ since $\alpha_i$ was shown to be bijective on $\frac{\sigma_{D_i}}{\Stab \sigma_{D_i}}$. This shows that $\Atrg = \coprod_{i=1}^k \alpha_i \left(\frac{\sigma_{D_i}}{\Stab \sigma_{D_i}}\right)$ as a set.

  Third, a face $F$ of some cone $\overline{\sigma_{D_i}}$ is $\overline{\sigma_{D(F)}}$, where $D(F)$ is a Delone subdivision that is a coarsening of $D_i$ \cite[Proposition 2.6.1]{v}. Then there exists $D_j$ and $A \in GL_g(\Z)$ with $\overline{\sigma_{D(F)}} \cdot A = \overline{\sigma_{D_j}}$ (recall that $A$ acts on a point $p \in \widetilde{S}^g_{\geq 0}$ by $p \mapsto A^TpA$). Restricting $A$ to the linear span of $\overline{\sigma_{D(F)}}$ gives a linear map
  \[ L_A: \sspan(\overline{\sigma_{D(F)}}) \longrightarrow \sspan (\overline{\sigma_{D_j}}) \]
  with the desired properties. Note, therefore, that $\overline{\sigma_{D_k}}$ is a stacky face of $\overline{\sigma_{D_i}}$ precisely if $D_k$ is $GL_g(\Z)$-equivalent to a coarsening of $D_i$.

  The fourth property then follows: the intersection
  \[ \alpha_i(\overline{\sigma_{D_i}}) \cap \alpha_j(\overline{\sigma_{D_j}}) = \bigcup \alpha_k(\sigma_{D_k}) \]
  where $\sigma_{D_k}$ ranges over all common stacky faces.
\end{proof}

\begin{proposition}
  The construction of $\Atrg$ in Definition \ref{d:ag} does not depend on our choice of $D_1, \dots, D_k$. More precisely, suppose $D_1', \dots, D_k'$ are another choice of representatives such that 
  $D_i'$ and $D_i$ are $GL_g(\Z)$-equivalent for each $i$. Let $A^{\text{tr}\,'}_g$ be the corresponding stacky fan. Then there is an isomorphism
  of stacky fans between $\, \Atrg \,$ and $\,A^{\text{tr}\,'}_g$.
\end{proposition}

\begin{proof}
  For each $i$, choose $A_i \in GL_g(\Z)$ with
  \[ \sigma_{D_i} \cdot A_i = \sigma_{D_i'}. \]
  Then we obtain a map
  \[ C(D_1) \coprod \cdots \coprod C(D_k) \xrightarrow{(A_1,\dots,A_k)} C(D_1') \coprod \cdots \coprod C(D_k') \]
  descending to a map
  \[ \Atrg \longrightarrow A^{\text{tr}'}_g,\]
  and this map is an isomorphism of stacky fans, as evidenced by the inverse map $A^{\text{tr}'}_g \rightarrow \Atrg$ constructed from the matrices $A_1^{-1}, \dots, A_k^{-1}$.
\end{proof}

\begin{theorem}\label{t:agh}
The moduli space $\Atrg$ is Hausdorff.
\end{theorem}

\begin{remark}
Theorem \ref{t:agh} complements the theorem of Caporaso that $\Mtrg$ is Hausdorff \cite[Theorem 5.2]{c}.
\end{remark}

\begin{proof}
Let $\overline{\sigma_{D_1}}, \dots, \overline{\sigma_{D_k}}$ be representatives for $GL_g(\Z)$-classes of secondary cones.  Let us regard $\Atrg$ as a quotient of the cones themselves, rather than the cones modulo their stabilizers, thus
\[ \Atrg =\left(\coprod\limits^k_{i=1} \overline{\sigma_{D_k}}\right)/\sim \]
where $\sim$ denotes $GL_g(\Z)$-equivalence as usual.  Denote by $\beta_i$ the natural maps 
$$\beta_i:\overline{\sigma_{D_i}}\longrightarrow \Atrg.$$
Now suppose $p\ne q \in \Atrg$.
For each $i = 1,\ldots,k$, pick 
disjoint open sets $U_i$ and $V_i$  in $\overline{\sigma_{D_i}}$ 
such that $\beta_i^{-1}(p)\subseteq U_i$ and $\beta_i^{-1}(q)\subseteq V_i$. Let
\begin{align*}
U&:=\{x\in\Atrg~:~\beta_i^{-1}(x)\subseteq U_i\textrm{ for all }i\},\\
V&:=\{x\in\Atrg~:~\beta_i^{-1}(x)\subseteq V_i\textrm{ for all }i\}.
\end{align*}
By construction, we have  $p\in U$ and $q\in V$.  We claim that $U$ and $V$ are disjoint open sets in $\Atrg$.

Suppose $x\in U\cap V$.  Now $\beta_i^{-1}(x)$ is nonempty for some $i$, hence $U_i\cap V_i$ is nonempty, contradiction.  Hence $U$ and $V$ are disjoint.
So we just need to prove that $U$ is open (similarly, $V$ is open).  It suffices to show that for each $j=1,\ldots,k$, the set $\beta_j^{-1}(U)$ is open.  Now,
\begin{align*}
\beta_j^{-1}(U)&=\{y\in \overline{\sigma_{D_j}}~:~\beta_i^{-1}\beta_j(y)\subseteq U_i\textrm{ for all }i\},\\
&=\bigcap_i \{y\in \overline{\sigma_{D_j}}~:~\beta_i^{-1}\beta_j(y)\subseteq U_i\}.
\end{align*}

Write $U_{ij}$ for the sets in the intersection above, so that $\beta_j^{-1}(U)=\bigcap_i U_{ij}$, and let $Z_i = \overline{\sigma_{D_i}}\setminus U_i$.  
Note that $U_{ij}$ consists of those points in $\overline{\sigma_{D_j}}$ that are not $GL_g(\Z)$-equivalent to any point in $Z_i$.
Then, just as in the proof of Theorem \ref{t:agsf}, there exist finitely many matrices $A_1,\ldots,A_s\in GL_g(\Z)$ such that
\begin{align*}
\overline{\sigma_{D_j}}\setminus U_{ij}&=\{y\in \overline{\sigma_{D_j}}~:~y\sim z\textrm{ for some }z\in Z_i\}\\
&= \bigcup_{l=1}^s \left(A_l^T Z_i A_l \cap \overline{\sigma_{D_j}}\right),
\end{align*}
which shows that $\overline{\sigma_{D_j}}\setminus U_{ij}$ is closed.  Thus the $U_{ij}$'s are open and so $\beta_j^{-1}(U)$ is open for each $j$.  Hence $U$ is open, and similarly, $V$ is open.
\end{proof}

\begin{remark}
Actually, we could have done a much more general construction of $\Atrg$.  We made a choice of decomposition of $\widetilde{S}^g_{\geq 0}$: we chose the second Voronoi decomposition, whose cones are secondary cones of Delone subdivisions.  This decomposition has the advantage that it interacts nicely with the Torelli map, as we will see.  But, as rightly pointed out in \cite{bmv}, we could use any decomposition of $\widetilde{S}^g_{\geq 0}$ that is ``$GL_g(\Z)$-admissible.'' This means that it is an infinite polyhedral subdivision of $\widetilde{S}^g_{\geq 0}$ such that $GL_g(\Z)$ permutes its open cones in a finite number of orbits.  See \cite[Section II]{amrt} for the formal definition.  Every result in this section can be restated for a general $GL_g(\Z)$-admissible decomposition: each such decomposition produces a moduli space which is a stacky fan, which is independent of any choice of representatives, and which is Hausdorff.  The proofs are all the same.  In this paper, though, we chose to fix a specific decomposition purely for the sake of concreteness and readability, invoking only what we needed to build up to the definition of the Torelli map.    
\end{remark}

\subsection{The quotient space $\widetilde{S}^g_{\geq 0} / GL_g(\Z)$}
We briefly remark on the construction of $\Atrg$ 
originally proposed in \cite{bmv}. There, the strategy is to start with the quotient space $\widetilde{S}^g_{\geq 0} / GL_g(\Z)$ and then hope to equip it directly with a stacky fan structure. But $\widetilde{S}^g_{\geq 0}/GL_g(\Z)$ is not even Hausdorff, as the following example shows.

\begin{example}
\label{e:cex}
Let $\{X_n\}_{n \geq 1}$ and $\{Y_n\}_{n \geq 1}$ be the sequences of matrices
\[ X_n = \left(\begin{matrix}
  1 & \frac{1}{n} \\ \frac{1}{n} & \frac{1}{n^2}
\end{matrix}\right), \
Y_n = \left(\begin{matrix}
  \frac{1}{n^2} & 0 \\ 0 & 0
\end{matrix}\right)\]
 in $\widetilde{S}^2_{\geq 0}$. Then we have
\[ \{X_n \} \rightarrow \left(\begin{matrix}
  1 & 0 \\ 0 & 0
\end{matrix}\right),\
\{Y_n \} \rightarrow \
 \left(\begin{matrix}
  0 & 0 \\ 0 & 0
\end{matrix}\right). \]
On the other hand, for each $n$, $X_n \equiv_{GL_2(\Z)} Y_n$ even while $\left(\begin{smallmatrix}1&0\\0&0\end{smallmatrix}\right) \not\equiv_{GL_2(\Z)} \left(\begin{smallmatrix}0&0\\0&0\end{smallmatrix}\right)$. This example then descends to non-Hausdorffness in the topological quotient. It can easily be generalized to $g > 2$.
\end{example}

In particular, we disagree with the claim in the proof of Theorem 4.2.4 of \cite{bmv} that the open cones $\sigma_D$, modulo their stabilizers, map homeomorphically onto their image in $\widetilde{S}^g_{\geq 0} / GL_g(\Z)$. Example \ref{e:cex} is a counterexample. However, we emphasize that our construction in Section 4.3 is just a minor modification of the ideas already extant in \cite{bmv}.

\section{Regular matroids and the zonotopal subfan}
In the previous section, we defined the moduli space $\Atrg$ of principally polarized tropical abelian varieties. In this section, we describe a particular stacky subfan of $\Atrg$ whose cells are in correspondence with simple regular matroids of rank at most $g$. This subfan is called the zonotopal subfan and denoted $\Azong$ because its cells correspond to those classes of Delone triangulations which are dual to zonotopes; see \cite[Section 4.4]{bmv}. The zonotopal subfan $\Azong$ is important because,
as we shall see in Section 6, it contains the image of the Torelli map. For $g \geq 4$,
this containment is proper.
 Our main contribution in this section is to characterize the stabilizing subgroups of all zonotopal cells.

We begin by recalling some basic facts about matroids.  A good reference is \cite{o}.
The connection between matroids and the Torelli map seems to have been first
observed by Gerritzen \cite{g}, and our approach here can be seen as an 
continuation of his work in the late 1970s.

\begin{definition}
  A matroid is said to be {\bf simple} if it has no loops and no parallel elements.
\end{definition}

\begin{definition}
  A matroid $M$ is {\bf regular} if it is representable over every field; equivalently, $M$ is regular if it is representable over $\R$ by a totally unimodular matrix. (A totally unimodular matrix is a matrix such that every square submatrix has determinant in $\{0,1,-1\}$.)
\end{definition}

Next, we review the correspondence between simple regular matroids and zonotopal cells. 

\begin{construction}\label{c:5}
Let $M$ be a simple regular matroid of rank at most $g$, and let $A$ be a $g \times n$ totally unimodular matrix that represents $M$. Let $v_1, \dots, v_n$ be the columns of $A$. Then let $\sigma_A \subseteq \R^{\binom{g+1}{2}}$ be the rational open polyhedral cone
\[ \R_{>0} \langle v_1v_1^T,\dots,v_nv_n^T\rangle. \]
\end{construction}

\begin{example}
Here is an example of Construction \ref{c:5} at work.
  Let $M$ be the uniform matroid $U_{2,3}$; equivalently $M$ is the graphic matroid $M(K_3)$. Then the $2 \times 3$ totally unimodular matrix
  \[ A = \left( \begin{matrix} 1 & 0 & 1 \\ 0 & 1 & -1 \end{matrix}\right) \]
  represents $M$, and $\sigma_A$ is the open cone generated by matrices
  \[ \left(\begin{matrix}1&0\\0&0\end{matrix}\right),
  \left(\begin{matrix}0&0\\0&1\end{matrix}\right),
  \left(\begin{matrix}1&-1\\-1&1\end{matrix}\right). \]
  It is the cone $\sigma_{D_1}$ in Example \ref{e:Atr2} and is shown in
  Figure \ref{f:a2_fan}.
\end{example}

\begin{proposition}\label{p:zon}
\cite[Lemma 4.4.3, Theorem 4.4.4]{bmv}
The cone $\sigma_A$ is a secondary cone in $\widetilde{S}^g_{\geq 0}$. Choosing a different totally unimodular matrix $A'$ to represent $M$ produces a cone $\sigma_{A'}$ that is $GL_g(\Z)$-equivalent to $\sigma_A$. Thus, we may associate to $M$ a unique cell of $\Atrg$, denoted $C(M)$.
\end{proposition}

\begin{definition}
  The {\bf zonotopal subfan} $\Azong$ is the union of cells
  \[ \{C(M):M \text{ a simple regular matroid of rank} \leq g \}
\quad \text{   in $\Atrg$.}\]
\end{definition}

We briefly recall the definition of the Voronoi polytope of a quadratic form in $\widetilde{S}^g_{\geq 0}$, just in order to explain the relationship with zonotopes.
\begin{definition}
  Let $Q \in \widetilde{S}^g_{\geq 0}$, and let $H = (\ker Q)^\bot \subseteq \R^g$. Then
  \[ \Vor(Q) = \{x \in H: x^TQx \leq (x-\lambda)^TQ(x-\lambda)\ \forall \lambda \in \Z^g\} \]
  is a polytope in $H\subseteq \R^g$, called the {\bf Voronoi polytope} of $Q$.
\end{definition}

\begin{theorem}\cite[Theorem 4.4.4, Definition 4.4.5]{bmv}
The zonotopal subfan
 $\Azong$ is a stacky subfan of $\Atrg$.  It consists of those points of 
 the tropical moduli space $\Atrg$ whose Voronoi polytope is a zonotope.
\end{theorem}

\begin{remark}
\label{r:perm}
Suppose $\sigma$ is an open rational polyhedral cone in $\R^n$. Then any $A \in GL_n(\Z)$ such that $A \sigma = \sigma$ must permute the rays of $\overline{\sigma}$, since the action of $A$ on $\overline{\sigma}$ is linear. Furthermore, it sends a first lattice point on a ray to another first lattice point; that is, it preserves lattice lengths. 
Thus, the subgroup $\Stab(\sigma) \subseteq GL_n(\Z)$ realizes some subgroup of the permutation group on the rays of $\overline{\sigma}$ (although if $\sigma$ is not full-dimensional then the action of $\Stab(\sigma)$ on its rays may not be faithful).
\end{remark}

Now, given a simple regular matroid $M$ of rank $\leq g$, we have almost computed the cell of $\Atrg$ to which it corresponds. Specifically, we have computed the cone $\overline{\sigma_A}$ for $A$ a matrix representing $M$, in Construction \ref{c:5}.  The remaining task is to compute the action of the stabilizer $\Stab(\sigma_A)$.

Note that $\overline{\sigma_A}$ has rays corresponding to the columns of $A$: a column vector $v_i$ corresponds to the ray generated by the symmetric 
rank $1$ matrix $v_iv_i^T$. In light of Remark \ref{r:perm}, we might conjecture that the permutations of rays of $\overline{\sigma_A}$ coming from the stabilizer are the ones that respect the matroid $M$, i.e.~come from matroid automorphisms. That is precisely the case and
provides valuable local information about~$\Atrg$.

\begin{theorem}
\label{t:aut}
Let $A$ be a $g \times n$ totally unimodular matrix representing the simple regular matroid $M$. 
Let $H$ denote the group of permutations of the rays of $\sigma_A$ which
are realized by the action of $Stab(\sigma_A)$.  Then $$H \cong\Aut(M).$$
\end{theorem}

\begin{remark}
This statement seems to have been known to Gerritzen in \cite{g}, but we present a new proof here, one which might be easier to read. Our main tool is the combinatorics of unimodular matrices.
\end{remark}

Here is a nice fact about totally unimodular matrices: they are essentially determined by the placement of their zeroes.

\begin{lemma}
\label{l:truemper}
\cite[Lemma 9.2.6]{truemper}
Suppose $A$ and $B$ are $g \times n$ totally unimodular matrices with the same support, i.e. $a_{ij} \neq 0$ if and only if $b_{ij} \neq 0$ for all $i, j$. Then $A$ can be transformed into $B$ by negating rows and negating columns.
\end{lemma}

\begin{lemma}
\label{l:unimodular}
Let $A$ and $B$ be $g \times n$ totally unimodular matrices, with column vectors $v_1, \dots, v_n$ and $w_1, \dots, w_n$ respectively. Suppose that the map $v_i \mapsto w_i$ induces an isomorphism of matroids $M[A] \overset{\cong}{\longrightarrow} M[B]$, i.e. takes independent sets to independent sets and dependent sets to dependent sets. Then there exists $X \in GL_g(\Z)$ such that
\[ Xv_i = \pm w_i, \text{ for each } i = 1,\dots, n. \]
\end{lemma}

\begin{proof}
  First, let $r = \rank(A) = \rank(B)$, noting that the ranks are equal since the matroids are isomorphic. Since the statement of Lemma \ref{l:unimodular} does not depend on the ordering of the columns, we may simultaneously reorder the columns of $A$ and the columns of $B$ and so assume that the first $r$ rows of $A$ (respectively $B$) form a basis of $M[A]$ (respectively $M[B]$). Furthermore, we may replace $A$ by $\Sigma A$ and $B$ by $\Sigma' B$, where $\Sigma,\Sigma' \in GL_g(\Z)$ are appropriate permutation matrices, and assume that the upper-left-most $r \times r$ submatrix of both $A$ and $B$ have nonzero determinant, in fact determinant $\pm 1$. Then, we can act further on $A$ and $B$ by elements of $GL_g(\Z)$ so that, without loss of generality, both $A$ and $B$ have the form
  \[ {\large \left[\begin{array}{rc|c}&\text{Id}_{r \times r}&*\\\hline &0&0\end{array}\right] }\]

  Note that after these operations, $A$ and $B$ are still totally unimodular; this follows from the fact that totally unimodular matrices are closed under multiplication and taking inverses. But then $A$ and $B$ are totally unimodular matrices with the same support. Indeed, the support of a column $v_i$ of $A$, for each $i = r+1, \dots, n$, is determined by the fundamental circuit of $v_i$ with respect to the basis $\{v_1, \dots, v_r\}$ in $M[A]$, and since $M[A] \cong M[B]$, each $v_i$ and $w_i$ have the same support.

  Thus, by Lemma \ref{l:truemper}, there exists a diagonal matrix $X \in GL_g(\Z)$, whose diagonal entries are $\pm 1$, such that $XA$ can be transformed into $B$ by a sequence of column negations. This is what we
  claimed.
  \end{proof}

\begin{proof}[Proof of Theorem \ref{t:aut}] Let $v_1, \dots, v_n$ be the columns of $A$. Let $X \in \Stab \sigma_A$. Then $X$ acts on the rays of $\overline{\sigma_A}$ via
\[ (v_i v_i^T)\cdot X = X^Tv_i v_i^T X = v_j v_j^T \text{ for some column } v_j. \]
So $v_j = \pm X^T v_i$. But $X^T$ is invertible, so a set of vectors $\{v_{i_1},\dots,v_{i_k}\}$ is linearly independent if and only if $\{X^Tv_{i_1},\dots,X^Tv_{i_k}\}$ is, so $X$ induces a permutation that is in $\Aut(M)$.

Conversely, suppose we are given $\pi \in \Aut(M)$. Let $B$ be the matrix
\[ B=\left[\begin{matrix}
  \ | & & |\  \\
  v_{\pi(1)} & \cdots & v_{\pi(n)} \\
  \ | & & |\
\end{matrix}\right]. \]

Then $M[A] = M[B]$, so by Lemma \ref{l:unimodular}, there exists $X \in GL_g(\Z)$ such that $X^T \cdot v_i = \pm v_{\pi(i)}$ for each $i$. Then $$X^Tv_i v_i^TX = (\pm v_{\pi(i)})(\pm {v_{\pi(i)}}^T) = v_{\pi(i)}{v_{\pi(i)}}^T$$ so $X$ realizes $\pi$ as a permutation of the rays of $\overline{\sigma_A}$.
\end{proof}

\section{The tropical Torelli map}
The classical Torelli map $t_g: \mathcal{M}_g \to \mathcal{A}_g$ sends a curve to its Jacobian.  Jacobians were developed thoroughly in the tropical setting in \cite{mz} and \cite{z}.  Here, we define the tropical Torelli map following \cite{bmv}, and recall the characterization of its image, the so-called Schottky locus, in terms of cographic matroids. We then present a comparison of the number of cells in $\Mtrg$, in the Schottky locus, and in $\Atrg$, for small $g$.

\begin{definition}
The tropical Torelli map
\[ \ttrg: \Mtrg \rightarrow \Atrg \]
is defined as follows.  Consider the first homology group
$H_1(G,\R)$ of the graph $G$, whose elements are formal sums of edges with coefficients in $\R$ lying in the kernel of the boundary map.
Given a genus $g$ tropical curve $C = (G,l,w)$,
we define a positive semidefinite form $Q_C$ on $H_1(G,\R)\oplus \R^{|w|}$, where $|w|:= \sum w(v)$.
The form is $0$ whenever the second summand $\R^{|w|}$ is involved, and on $H_1(G,\R)$ it is 
\[ Q_C(\sum_{e \in E(G)} \alpha_e \cdot e) \,\,\,= \,\,\,\sum_{e \in E(G)}\alpha_e^2 \cdot l(e). \]
Here, with the edges of $G$ are oriented for reference, the $\alpha_e$ are real numbers such that $\sum \alpha_e \cdot e\in H_1(G,\R)$. 

Now, pick a basis of $H_1(G,\Z)$; this identifies $H_1(G,\Z)\oplus \Z^{|w|}$ with the lattice $\Z^g$, and hence $H_1(G,\R)\oplus\R^{|w|}$ with $\R^g = \Z^g \otimes_\Z \R$. Thus $Q_C$ is identified with an element of $\widetilde{S}^g_{\geq 0}$. Choosing a different basis gives another element of $\widetilde{S}^g_{\geq 0}$ only up to a $GL_g(\Z)$-action, so we have produced a well-defined element of $\Atrg$, called the {\bf tropical Jacobian} of $C$.
\end{definition}

\begin{theorem}
  \cite[Theorem 5.1.5]{bmv}
  \label{t:morphism}
  The map
  \[ \ttrg: \Mtrg \to \Atrg \]
  is a morphism of stacky fans.
\end{theorem}

Note that the proof by Brannetti, Melo, and Viviani of Theorem \ref{t:morphism} is correct under the new definitions.  In particular, the definition of a morphism of stacky fans has not changed.  

The following theorem tells us how the tropical Torelli map behaves, at least on the level of stacky cells.
Given a graph $G$, its cographic matroid is denoted $M^*(G)$, and $\widetilde{M^*(G)}$ is then the matroid obtained by removing loops and replacing each parallel class with a single element.  See \cite[Definition 2.3.8]{bmv}.

\begin{figure}%
\includegraphics[width=1.5in]{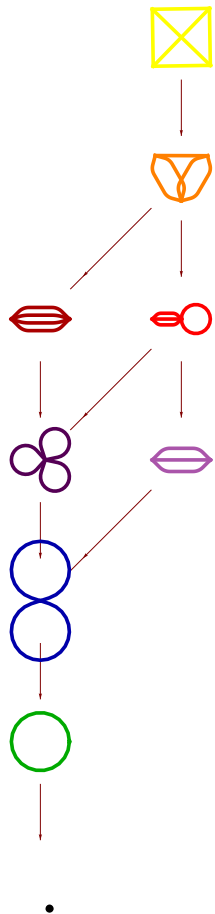}%
\vskip -.7in
\caption{Poset of cells of $\Atr{3}=\Acogr{3}$.  Each cell corresponds to a cographic matroid, and for convenience, we draw a graph $G$ in order to represent its cographic matroid $M^*(G)$.}%
\label{f:a_3}%
\end{figure}
\begin{theorem}\label{t:cogr}
\cite[Theorem 5.1.5]{bmv}
The map $\ttrg$ sends the cell $C(G,w)$ of $\Mtrg$ surjectively to the cell $C(\widetilde{M^*(G)})$. 
\end{theorem}
We denote by $\Acogrg$ the stacky subfan of $\Atrg$ consisting of those cells
\[ \{C(M): M \text{ a simple cographic matroid of rank } \leq g \}. \]
The cell $C(M)$ was defined in Construction \ref{c:5}. Note that $\Acogrg$ sits inside the zonotopal subfan of Section 5:
\[ \Acogrg \subseteq \Azong \subseteq \Atrg. \]
Also, $\Acogrg = \Atrg$ when $g \leq 3$, but not when $g \geq 4$ (\cite[Remark 5.2.5]{bmv}). The previous theorem says  that the image of $\ttrg$ is precisely $\Acogrg \subseteq \Atrg$. So, in analogy with the classical situation, we call $\Acogrg$ the {\bf tropical Schottky locus}.

Figures \ref{f:m_3} and \ref{f:a_3} illustrate the tropical Torelli map in genus 3.  The cells of $M^{\text{tr}}_3$ in Figure \ref{f:m_3} are color-coded according to the color of the cells of $A^{\text{tr}}_3$ in Figure \ref{f:a_3} to which they are sent.  These figures serve to illustrate the correspondence in Theorem \ref{t:cogr}.

Our contribution in this section is to compute the poset of cells of $\Acogrg$, for $g \leq 5$, using {\sc Mathematica}. First, we computed the cographic matroid of each graph of genus $\leq g$, and discarded the ones that were not simple. Then we checked whether any two matroids obtained in this way were in fact isomorphic. Part of this computation was done by hand in the genus 5 case, because it became intractable to check whether two 12-element matroids were isomorphic. Instead, we used some heuristic tests and then checked by hand that, for the few pairs of matroids passing the tests, the original pair of graphs were related by a sequence of vertex-cleavings and Whitney flips.
This condition ensures that they have the same cographic matroid; see \cite{o}.

\begin{theorem}
\label{t:sch}{\ } We obtained the following computational results:
\begin{enumerate}
\item
  The tropical Schottky locus $\Acogr{3}$ has nine cells and $f$-vector
  \[ (1,1,1,2,2,1,1). \]
Its poset of cells is shown in Figure \ref{f:a_3}.
\item  The tropical Schottky locus $\Acogr{4}$ has 25 cells and $f$-vector
  \[ (1,1,1,2,3,4,5,4,2,2). \]
\item
  The tropical Schottky locus $\Acogr{5}$ has 92 cells and $f$-vector
  \[ (1,1,1,2,3,5,9,12,15,17,15,7,4). \]
  \end{enumerate}
\end{theorem}

\begin{remark}
    Actually, since $\Acogr{3} = \Atr{3}$, the results of part (i) of Theorem \ref{t:sch} were already known, say in \cite{v}.
\end{remark}

Tables 1 and 2 show a comparison of the number of maximal cells and the number of total cells, respectively, of $\Mtrg$, $\Acogrg$, and $\Atrg$. The numbers in the first column of Table 2 were obtained in \cite{mp} and in Theorem \ref{t:comp}.  The first column of Table 1 is the work of Balaban \cite{b}. The results in the second column are our contribution in Theorem \ref{t:sch}.  The third columns are due to \cite{engel} and \cite{eg}; computations for $g > 5$ were done by Vallentin \cite{v}.

\begin{table}[h!]
\begin{tabular}{c|c|c|c}
$g$ & $\Mtrg$ & $\Acogrg$ & $\Atrg$ \\
\hline
2 & 2 & 1 & 1 \\
3 & 5 & 1 & 1 \\
4 & 17 & 2 & 3 \\
5 & 71 & 4 & 222
\end{tabular}
\caption{Number of maximal cells in the stacky fans $\Mtrg$, $\Acogrg$, and $\Atrg$.}
\end{table}
\begin{table}[h!]
\begin{tabular}{c|c|c|c}
$g$ & $\Mtrg$ & $\Acogrg$ & $\Atrg$ \\
\hline
2 & 7 & 4 & 4 \\
3 & 42 & 9 & 9 \\
4 & 379 & 25 & 61 \\
5 & 4555 & 92 & 179433
\end{tabular}
\caption{Total number of cells in the stacky fans $\Mtrg$, $\Acogrg$, and $\Atrg$.}
\end{table}

It would be desirable to extend our computations of $\Acogrg$ to $g \geq 6$, but
this would require some new ideas on effectively testing matroid isomorphisms.

\section{Tropical curves via level structure}
One problem with the spaces $\Mtrg$ and $\Atrg$ is that although they are tropical moduli spaces, they do not ``look'' very tropical: they do not satisfy a tropical balancing condition (see \cite{ms}). In other words: stacky fans, so far, are not tropical varieties. But what if we allow ourselves to consider finite-index covers of our spaces -- can we then produce a more tropical object? In what follows, we answer this question for the spaces $A^{\text{tr}}_2$  and $A^{\text{tr}}_3$. The uniform matroid $U^2_4$ and the Fano matroid $F_7$ play a role. We are grateful to Diane Maclagan for suggesting this question and the approach presented here.

Given $n \geq 1$, let $\mathbb{FP}^n$ denote the complete polyhedral fan in $\R^n$ associated to
projective space $\P^n$, regarded as a  toric variety.
 Concretely, we fix the rays of $\mathbb{FP}^n$ to be generated by
\[ e_1, \dots, e_n,\ \  e_{n+1} := -e_1-\dots-e_n, \]
and each subset of at most $n$ rays spans a cone in $\mathbb{FP}^n$. 
So $\mathbb{FP}^n$ has $n+1$ top-dimensional cones. Given $S \subseteq \{1,\dots,n+1\}$, let $\cone(S)$ denote the open cone $\R_{>0}\{e_i:i\in S\}$ in $\mathbb{FP}^n$, let $\cone(\hati):=\cone(\{1,\dots,\hati,\dots,n+1\})$, and let $\overline{\cone(S)}$ be the closed cone corresponding to $S$.
Note that the polyhedral fan $\mathbb{FP}^n$ is also a stacky fan: each open cone can be equipped with trivial symmetries. 

\subsection{A tropical cover for $\Atr{3}$}
By the classification in Sections 4.1--4.3 of \cite{v}, we note that
\[ \Atr{3} = \left(\coprod_{M \subseteq MK_4} C(M)\right) / \sim. \]
In the disjoint union above, the symbol $MK_4$ denotes the graphic (equivalently, in this case, cographic) matroid of the graph $K_4$, and $M \subseteq M'$ means that $M$ is a submatroid of $M'$, i.e. obtained by deleting elements. The cell $C(M)$ of a regular matroid $M$ was defined in Construction \ref{c:5}.  There is a single maximal cell $C(MK_4)$ in  $\Atr{3}$, and the other cells are stacky faces of it.  The cells are also listed in Figure~\ref{f:a_3}.

Now define a continuous map
\[ \pi: \mathbb{FP}^6 \to \Atr{3} \]
as follows. Let $A$ be a $3 \times 6$ unimodular matrix representing $MK_4$, for example
\[ A = \left(\begin{matrix}
  1 & 0 & 0 & 1 & 1 & 0 \\
  0 & 1 & 0 & -1 & 0 & 1 \\
  0 & 0 & 1 & 0 & -1 & -1
\end{matrix}\right), \]
and let $\overline{\sigma_A}$ be the cone in $\widetilde{S}^3_{\geq 0}$ with rays $\{v_iv_i^T\}$, where the $v_i$'s are the columns of $A$, as in Construction \ref{c:5}. Fix, once and for all, a 
Fano matroid structure on the set $\{1, \ldots, 7\}$.  For example, we could take $F_7$ to have circuits $\{124,235,346,457,156,267,137\}$.

Now, for each $i = 1, \dots, 7$, the deletion $F_7 \setminus \{i\}$ is isomorphic to $MK_4$, so let 
\[\pi_{\hati}: [7] \setminus \{i\} \rightarrow E(MK_4)\]
be any bijection inducing such an isomorphism. Now define
\[ \alpha_{\hati}: \overline{\cone(\hati)} \rightarrow \Atr{3} \]
as the composition
\[ \overline{\cone(\hati)} \xrightarrow{\ L_{\hati}\ } \overline{\sigma_A} \twoheadlongrightarrow \frac{\overline{\sigma_A}}{\Stab \sigma_A} = C(MK_4) \longrightarrow \Atr{3}\]
where $L_{\hati}$ is the integral-linear map arising from $\pi_\hati$.

Now, each $\alpha_\hati$ is clearly continuous, and to paste them together into a map on all of $\mathbb{FP}^6$, we need to show that they agree on intersections. Thus, fix $i \neq j$ and let $S \subseteq \{1,\dots,7\}\setminus\{i,j\}$.
We want to show that
\[ \alpha_\hati = \alpha_{\hatj} \text{ on } \overline{\cone(S)}. \]

Indeed, the map $L_\hati$ sends $\overline{\cone(S)}$ isomorphically to $\overline{\sigma_{A|_{\pi_\hati(S)}}}$, where $A|_{\pi_\hati(S)}$ denotes the submatrix of $A$ gotten by taking the columns indexed by $\pi_\hati(S)$. Furthermore, the bijection on the rays of the cones agrees with the isomorphism of matroids
\[ F_7|_S \xrightarrow{\ \ \cong\ \ } MK_4|_{\pi_\hati(S)}. \]
Similarly, $L_{\hatj}$ sends $\overline{\cone(S)}$ isomorphically to $\overline{\sigma_{A|_{\pi_{\hatj}(S)}}}$, and the map on rays agrees with the matroid isomorphism
\[ F_7|_S \xrightarrow{\ \ \cong\ \ } MK_4|_{\pi_\hatj(S)}. \]
Hence $MK_4|_{\pi_\hati(S)} \cong MK_4|_{\pi_\hatj(S)}$ and by Theorem \ref{t:aut}, there exists $X \in GL_3(\Z)$ such that the diagram commutes:
\[ \xymatrix{
\ & & \overline{\sigma_{A|_{\pi_\hati(S)}}} \ar[dd]^X\\
\overline{\cone(S)} \ar[rru]^{L_\hati}\ar[rrd]^{L_\hatj}& & \ \\
& & \overline{\sigma_{A|_{\pi_\hatj(S)}}} \\
} \]
We conclude that $\alpha_\hati$ and $\alpha_\hatj$ agree on $\overline{\cone(S)}$, since $L_\hati$ and $L_\hatj$ differ only by a $GL_3(\Z)$-action.

Therefore, we can glue the seven maps $\alpha_\hati$ together to obtain a continuous map $\alpha: \mathbb{FP}^6 \rightarrow \Atr{3}$.

\begin{theorem}
  \label{t:p6}
  The map $\alpha: \mathbb{FP}^6 \rightarrow \Atr{3}$ is a surjective morphism of stacky fans. Each of the seven maximal cells of $\mathbb{FP}^6$ is mapped surjectively onto the maximal cell of $\Atr{3}$.
\end{theorem}

\begin{proof}
  By construction, $\alpha$ sends each cell $\cone(S)$ of $\mathbb{FP}^6$ surjectively onto the cell of $\Atr{3}$ corresponding to the matroid $F_7|_S$, and each of these maps is induced by some integral-linear map $L_\hati$. That $\alpha$ is surjective then follows from the fact that every submatroid of $MK_4$ is a proper submatroid of $F_7$. Also, by construction, $\alpha$ maps each maximal cell $\cone(\hati)$ of $\mathbb{FP}^6$ surjectively to the cell $C(MK_4)$ of $\Atr{3}$.
\end{proof}

\subsection{A tropical cover for $\Atr{2}$}
Our strategy in Theorem \ref{t:p6} for constructing a covering map $\mathbb{FP}^6 \rightarrow \Atr{3}$ was to use the combinatorics of the Fano matroid to paste together seven copies of $MK_4$ in a coherent way. In fact, an analogous, and easier, argument yields a covering map $\mathbb{FP}^3 \to \Atr{2}$. We will use $U^2_4$ to paste together four copies of $U^2_3$.  Here, $U^d_n$ denotes the uniform rank $d$ matroid on $n$ elements.

The space $\Atr{2}$ can be given by
\[ \Atr{2} = \left(\coprod_{M \subseteq U^2_3} C(M)\right) / \sim. \]
It has a single maximal cell $C(U^2_3)$, and the three other cells are stacky faces of it of dimensions 0, 1, and 2.  See Figure~\ref{f:a2_fan}.

Analogously to Section 7.1, let
\[ A = \left(\begin{matrix} 1 & 0 & 1 \\ 0 & 1 & -1 \end{matrix}\right),\]
say, and for each $i = 1, \dots, 4$, define
\[ \beta_\hati: \overline{\cone(\hati)} \to \Atr{2} \]
by sending $\overline{\cone(\hati)}$ to $\overline{\sigma_A}$ by a bijective linear map preserving lattice points. Here, any of the $3!$ possible maps will do, because the matroid $U^2_3$ has full automorphisms.

Just as in Section 7.1, we may check that the four maps $\alpha_\hati$ agree on their overlaps, so we obtain a continuous map
\[ \beta: \mathbb{FP}^3 \rightarrow \Atr{2}. \]

\begin{proposition}
  The map $\beta: \mathbb{FP}^3 \rightarrow \Atr{2}$ is a surjective morphism of stacky fans. Each of the four maximal cells of $\mathbb{FP}^3$ maps surjectively onto the maximal cell of $\Atr{2}$.
\end{proposition}

\begin{proof}
  The proof is exactly analogous to the proof of Theorem \ref{t:p6}. Instead of noting that every one-element deletion of $F_7$ is isomorphic to $MK_4$, we make the easy observation that every one-element deletion of $U^2_4$ is isomorphic to $U^2_3$.
\end{proof}

\begin{remark}
  We do not know a more general construction for $g \geq 4$. We seem to be relying on the fact that all cells of $\Atrg$ are cographic when $g = 2,3$, but this is not true when $g \geq 4$: the Schottky locus is proper.
\end{remark}

\begin{remark}
  Although our constructions look purely matroidal, they come from level structures on $\Atr{2}$ and $\Atr{3}$ with respect to the primes $p = 3$ and $p = 2$, respectively. More precisely, in the genus 2 case, consider the decomposition of $\widetilde{S}^2_{\geq 0}$ into secondary cones as in Theorem \ref{t:main}, and identify rays $v v^T$ and $ww^T$ if $v \equiv \pm w \pmod{3}$. Then we obtain $\mathbb{FP}^3$. The analogous statement holds, replacing the prime 3 with 2, in genus 3.
\end{remark}

\bigskip
\end{document}